\documentclass[11pt,reqno]{amsart}

\usepackage{amssymb, amsmath, amsthm}
\usepackage[backref]{hyperref}
\usepackage[alphabetic,backrefs,lite]{amsrefs}
\usepackage{verbatim}
\usepackage{amscd}   
\usepackage[all]{xy} 
\usepackage{youngtab} 
\usepackage{young} 
\usepackage{tikz}

\newcommand{\defi}[1]{{\upshape\sffamily #1}}

\renewcommand{\1}{{\bf 1}}
\renewcommand{\a}{\alpha}

\newcommand{\A}{\mc{A}}
\renewcommand{\b}{\beta}
\newcommand{\bt}{\boxtimes}

\newcommand{\C}{\mathcal{C}}
\renewcommand{\d}{\underline{d}}

\newcommand{\FI}{\mathrm{FI}}
\renewcommand{\H}{\tilde{H}}

\newcommand{\K}{\bb{K}}
\renewcommand{\ll}{\lambda}
\newcommand{\m}{\mathfrak{m}}

\newcommand{\N}{\underline{N}}
\renewcommand{\o}{\otimes}
\newcommand{\oo}{\otimes}

\newcommand{\s}{\sigma}
\renewcommand{\S}{\mathfrak{S}}

\newcommand{\GL}{{GL}}
\newcommand{\Sym}{\operatorname{Sym}} 

\newcommand{\bb}[1]{\mathbb{#1}}
\renewcommand{\rm}[1]{\mathrm{#1}}
\newcommand{\mc}[1]{\mathcal{#1}}
\newcommand{\ul}[1]{\underline{#1}}

\newcommand{\ccirc}[1]{\xymatrix@1{*+<1ex>[o][F-]{#1}}}

\def\mone{-1}
\def\mtwo{-2}
\def\mthree{-3}

\def\lra{\longrightarrow}

\newtheorem{theorem}{Theorem}[section]
\newtheorem*{theorem*}{Theorem}
\newtheorem{lemma}[theorem]{Lemma}

\newtheorem{proposition}[theorem]{Proposition}
\newtheorem{corollary}[theorem]{Corollary}
\newtheorem*{corollary*}{Corollary}

\newtheorem*{corSTAB}{Corollary \ref{cor:stability}}
\newtheorem*{corNp}{Corollary \ref{cor:Np}}
\newtheorem*{corsyzstability}{Theorem \ref{cor:syzstability}}

\newtheorem*{tlinearstrand}{Theorem~\ref{thm:linearstrand}}

\theoremstyle{definition}
\newtheorem{definition}[theorem]{Definition}
\newtheorem*{definition*}{Definition}
\newtheorem{example}[theorem]{Example}

\theoremstyle{remark}
\newtheorem{remark}[theorem]{Remark}
\newtheorem*{remark*}{Remark}

\numberwithin{equation}{section}

\begin{document}

\title[Representation stability for syzygies]{Representation stability for syzygies of line bundles on Segre--Veronese varieties}

\author{Claudiu Raicu}
\address{Department of Mathematics, Princeton University, Princeton, NJ 08544-1000\newline
\indent Institute of Mathematics ``Simion Stoilow'' of the Romanian Academy}
\email{craicu@math.princeton.edu}

\subjclass[2010]{Primary 13D02, 14M12, 05E10, 55U10}

\date{\today}

\keywords{Syzygies, representation stability, Segre varieties, Veronese varieties, chessboard complexes, matching complexes, packing complexes, asymptotic vanishing}

\begin{abstract} The rational homology groups of the packing complexes are important in algebraic geometry since they control the syzygies of line bundles on projective embeddings of products of projective spaces (Segre--Veronese varieties). These complexes are a common generalization of the multidimensional chessboard complexes and of the matching complexes of complete uniform hypergraphs, whose study has been a topic of interest in combinatorial topology. We prove that the multivariate version of representation stability, a notion recently introduced and studied by Church and Farb, holds for the homology groups of packing complexes. This allows us to deduce stability properties for the syzygies of line bundles on Segre--Veronese varieties. We provide bounds for when stabilization occurs and show that these bounds are sometimes sharp by describing the linear syzygies for a family of line bundles on Segre varieties. 

As a motivation for our investigation, we show in an appendix that Ein and Lazarsfeld's conjecture on the asymptotic vanishing of syzygies of coherent sheaves on arbitrary projective varieties reduces to the case of line bundles on a product of (at most three) projective spaces.
\end{abstract}

\maketitle

\section{Introduction}

In this paper we prove that the rational homology groups of \defi{packing complexes} satisfy representation stability in the sense of Church and Farb, and we derive as a consequence a stabilization phenomenon for the syzygies of line bundles on Segre--Veronese varieties. Of particular interest is the case of ``stabilization to zero'', i.e. when the rational homology groups, respectively the syzygy modules, become trivial. The reason for this is explained in the appendix where we show that the conjecture of Ein and Lazarsfeld on the asymptotic vanishing of syzygies of sufficiently positive embeddings of a projective variety reduces to a vanishing statement for syzygies of line bundles on a product of (at most three) projective spaces.

We begin by formulating a theorem that illustrates the kind of syzygy stabilization results that we are aiming for. We first introduce some notation: when $X\subset\bb{P}W$ is a projective variety, embedded by the complete linear series corresponding to some line bundle $\mc{L}$, we associate to any sheaf $\mc{B}$ on $X$ the \defi{Koszul cohomology group $K_{p,q}(X,\mc{B};\mc{L})$} (Section \ref{subsec:syzfun}). If we let $B=\bigoplus_{n\in\bb{Z}}H^0(X,\mc{B}\oo\mc{L}^{\oo n})$ and $S=\Sym(W)$ then $K_{p,q}(X,\mc{B};\mc{L})$ is the space of minimal $p$--syzygies of degree $(p+q)$ of the $S$--module $B$.

\begin{tlinearstrand} For $n\geq 2$, we let $X=\bb{P}V_1\times\cdots\times\bb{P}V_n$, where $V_i$ are vector spaces over a field $\K$ of characteristic zero, and consider the line bundles $\mc{L}=\mc{O}(1,1,\cdots,1)$ and $\mc{B}_a=\mc{O}(a,0,\cdots,0)$ on $X$. For $p\geq 0$ and $\ll=(\ll^1,\cdots,\ll^n)$ a collection of partitions of $p$ we let $m_{\ll}$ denote the multiplicity of $S_{\ll^1}V_1\oo\cdots\oo S_{\ll^n}V_n$ inside $\bigwedge^p(V_1\oo\cdots\oo V_n)$. We have the decomposition into irreducible $\GL(V_1)\times\cdots\times\GL(V_n)$--representations
\[K_{p,0}(X,\mc{B}_a;\mc{L})=\bigoplus_{\ll} (S_{\ll^1[p+a]}V_1\oo S_{\ll^2}V_2\oo\cdots\oo S_{\ll^n}V_n)^{\oplus m_{\ll}},\]
where given a partition $\delta=(\delta_1,\delta_2,\cdots)$ of some integer $r$ we write $\delta[m]$ for the partition $(m-r,\delta_1,\delta_2,\cdots)$. $S_{\delta}$ denotes the Schur functor associated to $\delta$, and we make the convention that $S_{\delta[m]}$ is identically zero when $m-r<\delta_1$.
\end{tlinearstrand}

Note that the conclusion of the theorem remains true in the case $n=1$ if we replace $K_{p,0}(\mc{B}_a)$ with the $p$--th syzygy module of $\m^a$, where $\m$ is the homogeneous maximal ideal in the polynomial ring $S=\Sym(V)$: it is well--known (see \cite[Cor.~3.2]{buch-eis} or \cite[(1.a.10)]{green2}) that the minimal free resolution of $\m^a$ is given by
\[0\leftarrow\m^a\leftarrow S_a V\oo S(-a)\leftarrow S_{a,1}V\oo S(-a-1)\leftarrow S_{a,1^2}V\oo S(-a-2)\leftarrow\cdots\]
Theorem \ref{thm:linearstrand} was known in the case $n=2$ where in fact all the modules $K_{p,q}(\mc{B}_a)$ can be described explicitly (see \cites{friedman-hanlon,rei-rob} or \cite[Chapter~6]{weyman} for a more general story). We will prove Theorem \ref{thm:linearstrand} by applying the techniques of \cite{friedman-hanlon} involving combinatorial Laplacians.

The description of syzygies in Theorem \ref{thm:linearstrand} is fairly explicit, the only mystery being the calculation of the multiplicities $m_{\ll}$. This is known to be a complicated plethysm problem, and our theorem is meant to illustrate that the problem of computing syzygies even for simple modules supported on a product of projective spaces is in some sense equally difficult. An asymptotic measure of the complexity of the syzygies in the linear and quadratic strands ($K_{p,0}$ and $K_{p,1}$) for the Veronese varieties has been obtained by Fulger and Zhou \cite{fulger-zhou} by analyzing the number of distinct irreducible representations appearing in these syzygy modules, as well as the sum of their multiplicities. In Theorem \ref{thm:linearstrandVero} we provide a concrete illustration of their theory by describing the linear syzygies of $\mc{O}(1)$ under a Veronese embedding.

We view Theorem \ref{thm:linearstrand} as a stabilization result in the following way, which we'll be able to generalize further: for $a$ large enough ($a\geq p$) the number of irreducible representations (counted with multiplicities) appearing in the decomposition of $K_{p,0}(\mc{B}_a)$ stabilizes, and furthermore, there is a simple recipe to get the decomposition of $K_{p,0}(\mc{B}_{a+1})$ from that of $K_{p,0}(\mc{B}_a)$. We prove a similar statement for the syzygies of line bundles $\mc{B}_{\ul{b}}=\mc{O}(b_1,\cdots,b_n)$ on a product of projective spaces $X=\bb{P}V_1\times\cdots\times\bb{P}V_n$ with respect to an ample line bundle $\mc{L}=\mc{O}(d_1,\cdots,d_n)$:

\begin{corsyzstability} Consider $r<n$, a sequence $\d=(d_1,\cdots,d_n)$ of positive integers, and fix nonnegative integers $p,q$ and $b_{r+1},\cdots,b_n$ such that the inequality $b_j<d_j$ holds for at least one value of $j\in\{r+1,\cdots,n\}$. For integers $b_1,\cdots,b_r$ we let $N_i=(p+q)d_i+b_i$. There exist a finite number of $n$--tuples of partitions $\ll$ and corresponding multiplicities $m_{\ll}$ such that the decomposition
\[K_{p,q}(X,\mc{B}_{\ul{b}};\mc{L})=\bigoplus_{\ll} \left(S_{\ll^1[N_1]}V_1\oo\cdots\oo S_{\ll^r[N_r]}V_r\oo S_{\ll^{r+1}}V_{r+1}\oo\cdots\oo S_{\ll^n}V_n\right)^{\oplus m_{\ll}}\]
holds independently of $b_1,\cdots,b_r$ as long as $b_i\geq (p+q)d_i$ for $i=1,\cdots,r$.
\end{corsyzstability}
\noindent The condition of the existence of an index $j>r$ such that $b_j<d_j$ in the above result is not restrictive since $K_{p,q}(\mc{B}_{\ul{b}})=K_{p,q+1}(\mc{B}_{\ul{b}}\oo\mc{L}^{-1})=K_{p,q+1}(\mc{B}_{\ul{b}-\d})$. Letting $d_1=\cdots=d_n=1$ and $r=1$ in the above corollary yields the situation of Theorem \ref{thm:linearstrand} where the inequality $b_i\geq(p+q)d_i$ is in fact sharp ($i=1$, $b_1=a$, $d_1=1$, $q=0$, so the inequality becomes $a\geq p$). Unfortunately, we were not able to give a description of the multiplicities $m_{\ll}$ as in Theorem~\ref{thm:linearstrand}.

A natural question to ask is whether the conclusion of Theorem \ref{cor:syzstability} remains valid when $r=n$. The answer is positive and in fact it is not difficult to show that $K_{p,q}(\mc{B}_{\ul{b}})=0$ when all $b_i\gg 0$, so stabilization occurs in the most naive possible way. The best vanishing result for $K_{p,q}(\mc{B}_{\ul{b}})$ that we are aware of is 

\begin{corNp}
 Let $\d=(d_1,\cdots,d_n)$ be a sequence of positive integers and let $\ul{b}=(b_1,\cdots,b_n)$ be a sequence of arbitrary integers. We have $K_{p,2}(\mc{B}_{\ul{b}})=0$ for $p\leq\min\{d_i+b_i:i=1,\cdots,n\}$.
\end{corNp}

\noindent As we explain in Section \ref{subsec:correspondence} this is a consequence of \cite[Thm.~5.3]{ath}, or of standard Castelnuovo--Mumford regularity arguments. 

If we let $b_1=\cdots=b_n=0$ in Corollary \ref{cor:Np} then we get that the homogeneous coordinate ring of the Segre--Veronese variety corresponding to the embedding via $\mc{L}=\mc{O}(\d)$ satisfies the Green--Lazarsfeld property $N_p$ (introduced in \cite{gre-laz}) for $p\leq\min_i d_i$. This was proved in \cite{her-sch-smi} and strengthened to $p\leq\min_i (d_i+1)$ in \cite{bru-con-rom}. The aforementioned vanishing results are far from being sharp: Rubei proved that the coordinate ring of a Segre variety satisfies $N_p$ for $p\leq 3$ \cite{rubei}; the coordinate ring of the $d$--th Veronese embedding of $\bb{P}^2$ satisfies property $N_p$ for $p\leq 3d-3$ \cite{bir} and it was conjectured in \cite{ott-pao} that the same is true for embeddings of higher dimensional projective spaces. More general asymptotic vanishing conjectures have been formulated by Ein and Lazarsfeld for the syzygies of arbitrary varieties and in particular for Veronese varieties \cite{ein-laz}. In the Appendix we prove that asymptotic vanishing statements for arbitrary varieties can be reduced to the case of Segre--Veronese varieties, which motivates the desire to obtain good vanishing statements for the modules $K_{p,q}(\mc{B}_{\ul{b}})$.

To prove Theorem \ref{cor:syzstability} we show that representation stability (see Section \ref{sec:stability}) holds for \defi{packing complexes} (defined below), and then use \cite[Thm.~5.3]{kar-rei-wac} to translate between the syzygy modules $K_{p,q}(\mc{B}_{\ul{b}})$ and the homology groups of packing complexes. We defer the description of the correspondence between syzygies and the homology of packing complexes, as well as the technical definitions of representation stability to later sections, and focus on packing complexes for the rest of the introduction. We refer the reader to \cites{church-farb,church-ellenberg-farb} for an introduction to representation stability and to \cite{sam-snowden} for an equivalent notion and an extension of the structural theory. We point out that part of the motivation for \cite{sam-snowden} was earlier work by Snowden where certain finiteness properties for syzygies of Segre embeddings are established \cite{snowden}.

\begin{definition}[Packing complexes]\label{def:SVcomplexes} Consider $n$--tuples $\d=(d_1,\cdots,d_n)$ of positive integers, and $\A=(A_1,\cdots,A_n)$ of finite sets. Let $V$ be the set of $n$--tuples $\a=(\a_1,\cdots,\a_n)$, where $\a_i$ is a subset of $A_i$ of size $d_i$. The \defi{packing complex $\C_{\A}^{\d}$} is the simplicial complex whose $(r-1)$--simplices are subsets $\{\a^1,\cdots,\a^r\}\subset V$ where $\a^i_k$ is disjoint from $\a^j_k$ whenever $i\neq j$, for $1\leq i,j\leq r$, $1\leq k\leq n$. Note that for each $i$, the symmetric group $\S_{A_i}$ of permutations of the set $A_i$ acts on $\C_{\A}^{\d}$ and hence also on its homology groups. When $A_i=\{1,\cdots,N_i\}$ for some $n$--tuple $\N=(N_1,\cdots,N_n)$ of positive integers, we write $\C_{\N}^{\d}$ for the corresponding packing complex. It has an action of the product of symmetric groups $\S_{\N}=\S_{N_1}\times\cdots\times\S_{N_n}$.
\end{definition}


\begin{example}\label{ex:c22}
 For $n=2$, $d_1=d_2=1$ and $N_1=N_2=2$ the complex $\C_{(2,2)}^{(1,1)}$ is $1$--dimensional (it can be thought of as a simplicial complex classifying configurations of nonattacking rooks on a $2\times 2$ chessboard). It has four vertices $(1,1),(2,1),(1,2),(2,2)$, and two edges, as shown below:
\[
 \xymatrix{
  (1,1) \ar@{-}[d] & (1,2) \ar@{-}[d] \\
  (2,2) & (2,1) \\
}
\]
If we write $z_{(i,j)}$ for the homology class of the point $(i,j)$, then we get that the reduced homology group $\H_0(\C_{(2,2)}^{(1,1)})$ has a basis consisting of a single element $u=z_{(1,1)}-z_{(2,1)}$. We have that $z_{(1,1)}-z_{(2,2)}$ and $z_{(2,1)}-z_{(1,2)}$ are both zero, as they represent the boundaries of the two edges. To understand $\H_0(\C_{(2,2)}^{(1,1)})$ as a $\S_2\times \S_2$--module, we need to understand how the transpositions $\s_1$ and $\s_2$ in the two factors act on $u$. We have
\[\s_1 \cdot u = z_{(2,1)}-z_{(1,1)} = -u,\]
and
\[\s_2\cdot u = z_{(1,2)} - z_{(2,2)} = z_{(2,1)}-z_{(1,1)} = -u,\]
(where the middle equality uses $z_{(1,1)}=z_{(2,2)}$ and $z_{(2,1)}=z_{(1,2)}$). It follows that both $\s_1$ and $\s_2$ act by multiplication by $-1$, which means that $\H_0(\C_{(2,2)}^{(1,1)})$ is the tensor product of the sign representations of the two factors. The sign representation of $\S_2$ corresponds to the partition $(1,1)$, i.e. to the Young diagram $\Yvcentermath1\tiny\yng(1,1)$. Therefore we can write
\[\H_0(\C_{(2,2)}^{(1,1)}) = \Yvcentermath1\tiny\yng(1,1)\o\yng(1,1).\]
We will see in Theorem \ref{thm:main} that this calculation is equivalent to the fact that the degree two equations defining matrices of rank one (the $2$--factor Segre embedding) are spanned precisely by the $2\times 2$ minors of a generic matrix.
\end{example}

Before stating the main stabilization result for the homology groups of packing complexes (see Theorem \ref{thm:stability} for the more technical statement), we introduce some more notation: given a partition $\delta\vdash r$, we write $[\delta]$ for the corresponding irreducible representation of the symmetric group $\S_r$; $\H_k$ denotes the $k$--th reduced homology group with coefficients in the field $\K$.

\begin{corSTAB} For $k\geq -1$ and fixed values of the parameters $N_{r+1},\cdots,N_n$, there exist a finite number of $n$--tuples of partitions $\ll=(\ll^1,\cdots,\ll^n)$ and multiplicities $m_{\ll}>0$ such that the decomposition
\[\H_k\left(\C_{\N}^{\d}\right)=\bigoplus_{\ll} ([\ll^1[N_1]]\o\cdots\o[\ll^r[N_r]]\o[\ll^{r+1}]\o\cdots\o[\ll^n])^{\oplus m_{\ll}}\]
holds for $N_i\geq 2m\cdot d_i$, $i=1,\cdots,r$, where $m=\min\{\lfloor N_j/d_j\rfloor:j=r+1,\cdots,n\}$.
\end{corSTAB}

Packing complexes generalize the (multidimensional) \defi{chessboard complexes} (the case $d_1=d_2=\cdots=d_n=1$) and the \defi{matching complexes of complete graphs} (the case $n=1$ and $d_1=2$). The study of the integral homology and of the connectedness properties of these complexes has been a topic of interest in combinatorial topology that originated in \cite{bouc} (see \cites{blvz,ziegler,wachs,ath,sha-wachs}). The approach of relating syzygies to simplicial homology was used by Reiner and Roberts \cite{rei-rob} to give an independent proof and a generalization of the results of Lascoux and J\'ozefiak--Pragacz--Weyman \cites{lascoux,jpw} on the Betti numbers of the ideals of $2\times 2$--minors of generic matrices and generic symmetric matrices. A particularly beautiful determination of the rational homology of $2$--dimensional chessboard complexes was obtained by Friedman and Hanlon \cite{friedman-hanlon} using \defi{combinatorial Laplacians}. The corresponding calculation for matching complexes of complete graphs was subsequently obtained by Dong and Wachs \cite{dong-wachs}. 

The paper is organized as follows. In Section \ref{sec:prelim} we recall some basic facts from representation theory and introduce the syzygy functors whose stability properties we intend to study. We also describe the relationship between these functors and the reduced homology groups of packing complexes. In Section \ref{sec:stability} we introduce the basic notions of representation stability in the multivariate setting, following the univariate case described in \cites{church-farb,church,church-ellenberg-farb}. In Section \ref{sec:induction} we set up an inductive procedure for studying the homology of the packing complexes by exhibiting a long exact sequence that relates the reduced homology groups of several of these complexes. We prove representation stability for the homology groups of packing complexes in Section \ref{sec:SVstability}, based on the results in Sections \ref{sec:stability} and \ref{sec:induction}. We end with the calculation of the linear syzygies for a family of line bundles on Segre varieties using combinatorial Laplacians in Section \ref{sec:examples}. In the Appendix we show how the asymptotic vanishing conjecture of Ein and Lazarsfeld for syzygies of arbitrary varieties reduces to a vanishing statement for syzygies of line bundles on a product of at most three projective spaces.

\section{Preliminaries}\label{sec:prelim}

\subsection{Representation Theory}\label{sec:repthy}

For an introduction to the representation theory of general linear and symmetric groups, see~\cite{ful-har} and also \cite[Chapter~1, Appendix~A]{macdonald}. If $\mu=(\mu_1\geq\mu_2\geq\cdots)$ is a partition of $r$ (written $\mu\vdash r$, or $r=|\mu|$) and $W$ a vector space over a field $\K$ of characteristic zero, then $S_{\mu}W$ (resp. $[\mu]$) denotes the irreducible representation of the general linear group $GL(W)$ (resp. of the symmetric group $\mathfrak{S}_{r}$) corresponding to $\mu$. If $\mu=(r)$, then $S_{\mu}W$ is $\rm{Sym}^r(W)$ and $[\mu]$ is the trivial $\mathfrak{S}_{r}$--representation. The $GL(W)$-- (resp. $\mathfrak{S}_r$--) representations $U$ that we consider decompose as $U = \bigoplus_{\mu} U_{\mu}$ where $U_{\mu}\simeq (S_{\mu}W)^{m_{\mu}}$ (resp. $U_{\mu}\simeq[\mu]^{m_{\mu}}$) is the \defi{$\mu$--isotypic component} of $U$. We make the analogous definitions when we work over products of general linear (resp. symmetric) groups, replacing partitions by $n$--tuples of partitions (called \defi{$n$--partitions} and denoted by $\vdash^n$). We write $\S_{A}$ for the group of permutations of a set $A$, and $\S_{\A}=\S_{A_1}\times\cdots\times\S_{A_n}$ for an $n$--tuple $\A=(A_1,\cdots,A_n)$ of sets. $\S_{\A}$ is isomorphic to the group $\S_{\N}=\S_{N_1}\times\cdots\times\S_{N_n}$ associated to the $n$--tuple $\N=(N_1,\cdots,N_n)$, where $N_i=|A_i|$. If $\ll\vdash^n\N$, $\ll=(\ll^1,\cdots,\ll^n)$, we write $S_{\ll}$ for the tensor product of Schur functors $S_{\ll^1}\oo\cdots\oo S_{\ll^n}$, and $[\ll]$ for the irreducible $\S_{\N}$--representation $[\ll^1]\oo\cdots\oo[\ll^n]$.

Given $n$--tuples $\N=(N_1,\cdots,N_n)$ and $\N'$, we say that $\N'$ is a \defi{successor} of $\N$ (or $\N$ a \defi{predecessor} of $\N'$, or that $\N,\N'$ are \defi{consecutive}) if $N_i' = N_i + 1$ for some $i$, and $N_j=N_j'$ for $j\neq i$. In general we write $\N\leq\N'$ if $N_i\leq N_i'$ for all $i$.

Following the notation in \cite{church-farb}, if $\ll$ is an $n$--partition, we write $\ll[\N]$ for the $n$--partition $\tilde{\ll}\vdash^n\N$ defined by $\tilde{\ll}^i=(N_i-|\ll^i|,\ll^i_1,\ll^i_2,\cdots)$ (of course this makes sense only if $N_i\geq|\ll^i|+\ll^i_1$). For instance, when $n=2$, $\ll=((3,1),(2,2,1))$ and $\N=(8,7)$, we have $|\ll^1|=4$, $|\ll^2|=5$, and $\ll[\N]=((4,3,1),(2,2,2,1))$. We will often picture $n$--partitions as formal tensor powers of Young diagrams, and interpret them according to the context as either irreducible representations of a product of general linear groups, or of a product of symmetric groups:
\[\Yvcentermath1\tiny\ll=\yng(3,1)\o\yng(2,2,1),\quad\ll[\N]=\yng(4,3,1)\o\yng(2,2,2,1).\]
Note that for $\N=(8,6)$, the $2$--partition $\ll[\N]$ is not defined.

If $U_i$ is a $G_i$--representation, $i=1,2$, for some groups $G_1,G_2$, then the external tensor product $U_1\bt U_2$ is a $G_1\times G_2$--representation (note that whenever we will try to emphasize the distinction between external and internal tensor products, we'll be using the symbol $\bt$ instead of $\o$). We write $\1_G$ (or just $\1$) for the trivial representation of a group $G$. For a subgroup $H\subset G$ and representations $U$ of $H$ and $W$ of $G$, we write
\[\rm{Ind}_H^G(U)=K[G]\o_{K[H]} U\ \rm{ and }\ \rm{Res}_H^G(W)=W_H\]
for the \defi{induced representation} of $U$ and the \defi{restricted representation} of $W$ respectively, where $K[M]$ denotes the group algebra of a group $M$, and $W_H$ is just $W$, regarded as an $H$--module. 

\subsection{The syzygy functors $K_{p,q}^{\d}(\ul{b})$}\label{subsec:syzfun} If $X\subset\bb{P}W$ is a projective variety, embedded by the complete linear series corresponding to some line bundle $\mc{L}$ (so that $W=H^0(X,\mc{L})$), we associate to any sheaf $\mc{B}$ on $X$ the \defi{Koszul cohomology group $K_{p,q}(X,\mc{B};\mc{L})$} (or simply $K_{p,q}(\mc{B})$ when $X$ and $\mc{L}$ are understood from the context) defined as the homology of the $3$--term complex
\begin{equation}\label{eq:koszul}
\bigwedge^{p+1}W\oo H^0(X,\mc{B}\oo\mc{L}^{q-1})\to\bigwedge^{p}W\oo H^0(X,\mc{B}\oo\mc{L}^{q})\to\bigwedge^{p-1}W\oo H^0(X,\mc{B}\oo\mc{L}^{q+1}) 
\end{equation}

Consider now the case when $X=\bb{P}V_1\times\cdots\times\bb{P}V_n$ is a product of projective spaces and $\mc{L}=\mc{O}(d_1,\cdots,d_n)$ is an ample line bundle on $X$. Write $\mc{B}_{\ul{b}}=\mc{O}(b_1,\cdots,b_n)$ for arbitrary integers $b_i$. It is clear that $X,\mc{L},\mc{B}_{\ul{b}}$ depend functorially of the vector spaces $V_1,\cdots,V_n$, thus the same is true about the Koszul cohomology groups $K_{p,q}(\mc{B}_{\ul{b}})$. We write $K_{p,q}^{\d}(\ul{b}):Vec^n\to Vec$ for the functor on finite dimensional $\K$--vector spaces that assigns to an $n$--tuple $(V_1,\cdots,V_n)$ the corresponding syzygy module $K_{p,q}(\mc{B}_{\ul{b}})$. As we will see in Theorem \ref{thm:main}, these functors are controlled by the homology of the packing complexes introduced in Definition \ref{def:SVcomplexes}. Figure \ref{fig:2factorsyz} below describes the beginning of the \defi{equivariant Betti table} $(K_{p,q}^{\d})$ for $\d=(1,1)$ (corresponding to the two--factor Segre embedding): dashes correspond to $K_{p,q}^{\d}=0$, and instead of writing $S_{\ll^1}\o S_{\ll^2}$, we picture the appropriate diagrams.
\begin{figure}[h]
\[
 \begin{array}{|c|c|c|c|c|c}
  \hline
  \K & - & - & - & - & \cdots\\
  \hline
   & & & & & \\
   &  & \tiny\Yvcentermath1\yng(2,1)\o\yng(1,1,1) & \tiny\Yvcentermath1\yng(3,1)\o\yng(1,1,1,1) + \yng(2,1,1)\o\yng(2,1,1) & \tiny\Yvcentermath1\yng(4,1)\o\yng(1,1,1,1,1) + \yng(3,1,1)\o\yng(2,1,1,1) & \\
  - & \tiny\Yvcentermath1\yng(1,1)\o\yng(1,1) & + & + & + & \cdots \\
   &  & \tiny\Yvcentermath1\yng(1,1,1)\o\yng(2,1) & \tiny\Yvcentermath1\yng(1,1,1,1)\o\yng(3,1) & \tiny\Yvcentermath1\yng(2,1,1,1)\o\yng(3,1,1) + \yng(1,1,1,1,1)\o\yng(4,1) \\ 
   & & & & &\\
  \hline
   & & & & &\\
  - & - & - & - & \tiny\Yvcentermath1\yng(2,2,2)\o\yng(2,2,2) & \cdots\\
   & & & & & \\
  \hline
  \vdots & \vdots & \vdots & \vdots & \vdots & \ddots
 \end{array}
\]
\caption{Syzygy functors for two--factor Segre embeddings}
\label{fig:2factorsyz}
\end{figure}

\subsection{The correspondence between syzygy functors and the homology of packing complexes}\label{subsec:correspondence}

In this section we describe the correspondence between the syzygy functors from the previous section and the (reduced) homology groups of the packing complexes introduced in Definition \ref{def:SVcomplexes}. This correspondence has been exploited by Reiner and Roberts \cite{rei-rob} to compute the syzygy functors for the quadratic Veronese and $2$--factor Segre varieties. It is an instance of more general results that relate syzygies of graded modules over affine semigroup rings to simplicial homology (\cite{bru-her}, \cite[Thm.~7.9]{sta}, \cite[Thm.~12.12]{stu}).

\begin{theorem}[{\cite[Thm.~5.3]{kar-rei-wac}}]\label{thm:main}
 Let $p,q$ be nonnegative integers, let $\d=(d_1,\cdots,d_n)$ be a sequence of positive integers, and let $\ul{b}=(b_1,\cdots,b_n)$ be a sequence of arbitrary integers. Write $N_i=(p+q)\cdot d_i+b_i$, and let $\N=(N_1,\cdots,N_n)$. Consider an $n$--partition $\ll\vdash^n\N$. The multiplicity of $S_{\ll}$ in $K_{p,q}^{\d}(\ul{b})$ coincides with the multiplicity of the irreducible $\S_{\N}$--representation $[\ll]$ in $\H_{p-1}(\C_{\N}^{\d})$.
\end{theorem}

We point out a vanishing result for the homology of packing complexes, which via the above theorem yields the vanishing of certain syzygy functors. We note that Theorem \ref{thm:vanishing} below in fact holds for integral homology, and that it would be desirable from the point of view of algebraic geometry to obtain sharper vanishing results for the rational homology of packing complexes.

\begin{theorem}[{\cite[Thm.~5.3]{ath}}]\label{thm:vanishing}
 Let $\d=(d_1,\cdots,d_n)$ be a sequence of positive integers and let $p\geq 0$. If $\N=(N_1,\cdots,N_n)$ with $N_i\geq p\cdot(d_i+1)+d_i$, $i=1,\cdots,n$, then $\H_{p-1}(\C_{\N}^{\d})=0$.
\end{theorem}

\begin{corollary}\label{cor:Np}
 Let $\d=(d_1,\cdots,d_n)$ be a sequence of positive integers and let $\ul{b}=(b_1,\cdots,b_n)$ be a sequence of arbitrary integers. We have $K_{p,2}^{\d}(\ul{b})=0$ for $p\leq\min\{d_i+b_i:i=1,\cdots,n\}$.
\end{corollary}

\begin{proof} The condition $K_{p,2}^{\d}(\ul{b})=0$ is equivalent via Theorem \ref{thm:main} to the vanishing of $\H_{p-1}(\C_{\N}^{\d})$, where $N_i=(p+2)\cdot d_i+b_i$. Applying Theorem \ref{thm:vanishing} we get that this vanishing holds as soon as $(p+2)\cdot d_i+b_i\geq p\cdot(d_i+1)+d_i$, which is equivalent to $d_i+b_i\geq p$. 

Alternatively, with the notation in Section \ref{subsec:syzfun} we have by \cite[Prop~3.2]{ein-laz} that
\begin{equation}\label{eq:kpqH1}
K_{p,2}^{\d}(\mc{B}_{\ul{b}})=H^1(X,\bigwedge^{p+1}M\oo\mc{O}_X(d_1+b_1,\cdots,d_n+b_n)),
\end{equation}
where $M=\rm{Ker}(H^0(X,\mc{O}_X(\d))\oo\mc{O}_X\to\mc{O}_X(\d))$ is the restricted tautological bundle corresponding to the embedding of $X$ by $\mc{O}(\d)$. We have $M=M_1\boxtimes M_2\boxtimes\cdots\boxtimes M_n$, where $M_i=\rm{Ker}(H^0(\bb{P}V_i,\mc{O}_{\bb{P}V_i}(d_i))\oo\mc{O}_{\bb{P}V_i}\to\mc{O}_{\bb{P}V_i}(d_i))$, so $\bigwedge^{p+1}M$ decomposes as a direct sum of $S_{\ll^1}M_1\boxtimes\cdots\boxtimes S_{\ll^n}M_n$ for $\ll^i\vdash (p+1)$. Using K\"unneth's formula, the vanishing of the terms in (\ref{eq:kpqH1}) reduces to proving that 
\[H^1(\bb{P}V_i,S_{\mu} M_i\oo\mc{O}_{\bb{P}V_i}(d_i+b_i))=0,\ \rm{for}\ \mu\vdash(p+1).\]
Now since $M_i$ is $1$--regular with respect to $\mc{O}_{\bb{P}V_i}(1)$ (see \cite[Section I.8]{lazarsfeld} for definitions), it follows that $M_i^{\oo(p+1)}$ is $(p+1)$--regular, hence the same is true about $S_{\mu}M_i$ which is a direct summand in $M_i^{\oo(p+1)}$. If $p\leq (d_i+b_i)$ then $S_{\mu}M_i$ is also $(d_i+b_i+1)$--regular and the desired vanishing follows.
\end{proof}

\section{Representation stability}\label{sec:stability}

This section is based on \cites{church-farb,church-ellenberg-farb}. We adopt a slightly different strategy from \cite{church-ellenberg-farb} which is valid only in characteristic zero, but offers a quick access to stability for the problem at hand, namely for the stabilization of homology of packing complexes.

We denote by $Set$ the category of sets, where morphisms are injective maps. For a positive integer $n$, we let $Set^n$ denote the $n$--fold product of $Set$ with itself. We write $Vec$ for the category of finite dimensional vector spaces over $\K$.

\begin{definition}[$\FI^n$--modules \cite{church-ellenberg-farb}] We define an \defi{$\FI^n$--module} to be a functor $V:Set^n\to Vec$. A \defi{morphism of $\FI^n$--modules} is just a natural transformation $T:V\to W$. We will often refer to $V$ as an \defi{$\FI$--module} or simply a \defi{module}, when there's no danger of confusion.
\end{definition}

If $V$ is an $\FI^n$--module, and $\A=(A_1,\cdots,A_n)$ is an $n$--tuple, then $V_{\A}$ admits a natural action of the product of symmetric groups $\S_{\A}=\S_{A_1}\times\cdots\times\S_{A_n}$. We can then think of the $\FI^n$--module $V$ as a pair $(V,\phi)$ consisting of a collection $V=(V_{\N})_{\N}$ of finite dimensional $\S_{\N}$--representations $V_{\N}$, indexed by $n$--tuples $\N=(N_1,\cdots,N_n)$ of positive integers, equipped with maps $\phi_{\N,\N'}:V_{\N}\to V_{\N'}$ for all consecutive $n$--tuples $\N,\N'$. These maps have to be equivariant with respect to the $\S_{\N}$--action when we regard $\S_{\N}$ as a subgroup of $\S_{\N'}$ in the natural way, i.e. we can think of $\phi_{\N,\N'}$ as a $\S_{\N}$--equivariant map $V_{\N}\to\rm{Res}_{\S_{\N}}^{\S_{\N'}}(V_{\N'})$, or a $\S_{\N'}$--equivariant map $\rm{Ind}_{\S_{\N}}^{\S_{\N'}}(V_{\N})\to V_{\N'}$. A morphism $T$ between $V=(V,\phi)$ and $W=(W,\psi)$ is then a collection of $\S_{\N}$--equivariant maps $T_{\N}:V_{\N}\to W_{\N}$, satisfying $\psi_{\N,\N'}\circ T_{\N}=T_{\N'}\circ\phi_{\N,\N'}$.
 
By composing maps between consecutive $n$--tuples we get maps $\phi_{\N,\N'}$ whenever $\N\leq\N'$ (i.e. $N_i\leq N_i'$ for all $i$). As remarked in \cite[Lemma~2.1, Remark~2.2]{church-ellenberg-farb}, $(V,\phi)$ needs to satisfy a further compatibility relation: denoting by $[\N]$ the set $\{1,\cdots,N_1\}\times\cdots\times\{1,\cdots,N_n\}$, we must have that for every $\N\leq\N'$, $v\in V_{\N}$ and $v'=\phi_{\N,\N'}(v)$, and for every $\s^1,\s^2\in\S_{\N'}$ such that $\s^1|_{[\N]}=\s^2|_{[\N']}$, the equality $\s^1(v')=\s^2(v')$ holds. 

The following definition of stability is inspired by \cite[Definition~1.2]{church}.

\begin{definition}[Representation stability]\label{def:repstability} The $\FI^n$--module $V$ is called \defi{representation stable} if for all $n$--partitions $\ll$ and all $\N\gg 0$ (i.e. for sufficiently large values of the parameters $N_1,\cdots,N_n$), the natural map (induced by $\phi_{\N,\N'}$)
\[\phi_{\N,\N'}(\ll):\left(\rm{Ind}_{\S_{\N}}^{\S_{\N'}}\left((V_{\N})_{\ll[\N]}\right)\right)_{\ll[\N']}\lra (V_{\N'})_{\ll[\N']}\]
is an isomorphism for all $\N'\geq\N$. We will often refer to $V$ as a \defi{stable} module, for simplicity. We say that $V$ has \defi{injectivity range/surjectivity range/stable range $\N'\geq\N$} if the maps $\phi_{\N,\N'}(\ll)$ are injective/surjective/isomorphisms for all $\ll$ whenever $\N'\geq\N$.
\end{definition}

Note that for $\N\gg 0$ and $\N'\geq\N$, the above definition implies that for a stable module $V$ the maps $\phi_{\N,\N'}$ are injective, the image of $\phi_{\N,\N'}$ generates $V_{\N'}$ as a $\S_{\N'}$--representation, and moreover, the multiplicity of $\ll[\N]$ inside $V_{\N}$ is independent of $\N$ for every $n$--partition $\ll$. This means that $V$ satisfies \defi{uniform representation stability} in the sense of \cite[Definition~2.6]{church-farb}.

For $0\leq s\leq n$ and a subset $I=\{i_1,\cdots,i_s\}$ of $\{1,\cdots,n\}$, we consider a fixed collection of finite sets $A_{i_1},\cdots,A_{i_s}$. Given any $\FI^n$--module $V$, we can restrict it to an $\FI^{n-s}$--module $W$, by letting $W((A_j)_{j\notin I})=V(A_1,\cdots,A_n)$. We call $W$ a \defi{restriction (pull--back)} of $V$.

\begin{definition}[Representation superstability]\label{def:superstability} The $\FI^n$--module $V$ is called \defi{representation superstable} if all its restrictions are representation stable. We will often refer to $V$ as a \defi{superstable} module, for simplicity.
\end{definition}

\begin{remark}\label{rem:superstability} For any $n\geq 1$, it makes sense to talk about finitely generated $\FI^n$--modules in the sense of \cite{church-ellenberg-farb}, or about finitely generated $\GL_{\infty}^n$--equivariant $\Sym((\bb{C}^{\infty})^n)$--modules in the sense of \cite{sam-snowden}. It can be checked that (in characteristic zero) a module is finitely generated if and only if it is superstable. 
\end{remark}

\begin{remark}[$\FI$--spaces]\label{rem:fispace} In the terminology of \cite{church-ellenberg-farb}, the functor that assigns to a tuple $\A$ of sets the packing complex $\C_{\A}^{\d}$ is an \defi{$\FI$--space}. Applying the reduced homology functors $\H_i$ to this $\FI$--space yields $\FI$--modules that are superstable (see Theorem \ref{thm:stability}).
\end{remark}

\begin{lemma}\label{lem:subquotstable} If $V$ is representation (super)stable and $W$ is a sub-- or quotient module of $V$, then $W$ is also representation (super)stable. More generally, if $V$ has a finite filtration with quotients $W_i$, then all $W_i$ are (super)stable if $V$ is.
\end{lemma}

\begin{proof} The superstable case is a consequence of the stable case, so we only deal with the latter. Since $V$ is a stable $\FI^n$--module, there are finitely many $n$--partitions $\ll$ such that $\ll[\N]$ appears in $V_{\N}$ for $\N\gg 0$ and moreover, the multiplicity $m_{\ll}(V_{\N})$ of $\ll[\N]$ in $V_{\N}$ is constant for $\N\gg 0$. If $W$ is a sub-- (resp. quotient) module of $V$, then for each such $\ll$ the induced maps $\phi_{\N,\N'}(\ll)|_W$ are injective (resp. surjective) for $\N'\geq\N\gg 0$, so the multiplicities $m_{\ll}(W_{\N})$ are eventually nondecreasing (resp. nonincreasing), hence they stabilize and therefore $\phi_{\N,\N'}(\ll)|_W$ are eventually bijective. The last statement follows by an easy induction.
\end{proof}
  
\begin{corollary}\label{cor:imkerstable} If $V,W$ are representation (super)stable, and $T:V\to W$ is a morphism then $\rm{Im}(T)$ and $\rm{Ker}(T)$ are also representation (super)stable.
\end{corollary}

\begin{lemma}\label{lem:trianglestable} Given an ``exact triangle'' $X_{\bullet}\to Y_{\bullet}\to Z_{\bullet}\to X_{\bullet}[-1]$, i.e. an exact sequence
\[\cdots\to X_k\to Y_k\to Z_k\to X_{k-1}\to Y_{k-1}\to\cdots\]
with $X_k$, $Y_k$ representation (super)stable for all $k$, then $Z_k$ is also representation (super)stable for every $k$. If $X_k,Y_k,X_{k-1},Y_{k-1}$ have stable range $\N'\geq\N$, then $Z_k$ also has stable range $\N'\geq\N$.

In particular, if
\[0\to A\to B\to C\to 0\]
is a short exact sequence of $\FI^n$--modules, and if any two of $A,B,C$ are (super)stable, then the same is true about the third. If $B$ has stable range $\N'\geq\N$ then $A$ has injectivity range $\N'\geq\N$ and $C$ has surjectivity range $\N'\geq\N$. If any two of $A,B,C$ have stable range $\N'\geq\N$ then the same is true about the third.
\end{lemma}

\begin{proof} Follows from the $5$--lemma.
\end{proof}

We say that an $\FI^n$--module $V$ is \defi{trivial} if $V_{\N}=0$ for $\N\gg 0$. It is \defi{supertrivial} if $V_{\N}=0$ except maybe for finitely many tuples $\N$. We note that a (super)trivial module is (super)stable. For the purpose of stability, it will be convenient to identify modules that coincide for sufficiently large multidegrees. More precisely, we say that $V$ and $W$ are \defi{equivalent} if there exist trivial submodules $V^0\subset V$, $W^0\subset W$, and an isomorphism between $V/V^0$ and $W/W^0$. We say that $V$ is \defi{simple} if it is trivial, or if it is equivalent to $W$ for every nontrivial submodule $W$ of $V$.

We denote by $V(\ll)$ the $\FI^n$--module where $V(\ll)_{\N}=[\ll[\N]]$ for all $\N$ for which $\ll[\N]$ is defined and $V(\ll)_{\N}=0$ otherwise, and for a successor $\N'$ of $\N$, the map $\phi_{\N,\N'}$ is zero when $V(\ll)_{\N}=0$, and otherwise it is the unique (up to scaling) nonzero $\S_{\N}$--equivariant map from $[\ll[\N]]$ to $[\ll[\N']]$. It is clear that $V(\ll)$ is simple and stable (in fact it is even superstable), and the next lemma shows that every simple stable module is equivalent to $V(\ll)$ for some $\ll$.

\begin{lemma}\label{lem:simplestable} If $V$ if a stable nontrivial module, then $V$ contains a submodule equivalent to $V(\ll)$ for some $n$--partition $\ll$. 
\end{lemma}

\begin{remark}\label{rem:filtrationstable} An easy induction argument combined with the above lemma shows that every stable module $V$ has a finite filtration (a \defi{composition series}) whose quotients are simple modules equivalent to $V(\ll)$ for $\ll$ in some finite collection $\mc{P}$ of $n$--partitions. We call each $\ll\in\mc{P}$ a \defi{constituent} of $V$. For each such $\ll$, we denote by $m_{\ll}$ the number of occurrences of (a module equivalent to) $V(\ll)$ in a composition series for $V$. We call $m_{\ll}$ the multiplicity of the constituent $\ll$. The constituents and their multiplicities are characterized by the decomposition
\[V_{\N}=\bigoplus_{\ll\in\mc{P}} [\ll[\N]]^{\oplus m_{\ll}}\ \rm{for}\ \N\gg 0.\]
\end{remark}

\begin{definition}\label{def:sizemaximal}
Given a collection $\mc{P}$ of $n$--partitions, we say that $\ll\in\mc{P}$ is \defi{size maximal} if for any $\tilde{\ll}\in\mc{P}$, we either have $|\ll^i|=|\tilde{\ll}^i|$ for all $i=1,\cdots,n$, or $|\ll^i|>|\tilde{\ll}^i|$ for some $i$. 
\end{definition}

\begin{lemma}\label{lem:sizemaximal} If $\ll\in\mc{P}$ is size maximal, and $\N,\N'$ are consecutive $n$--tuples, then for every $\tilde{\ll}\in\mc{P}$ different from $\ll$, $\ll[\N]$ does not occur in $\rm{Res}_{\S_{\N}}^{\S_{\N'}}\left(\tilde{\ll}[\N']\right)$. 
\end{lemma}

\begin{proof} This follows from Pieri's rule. 
\end{proof}

\begin{proof}[Proof of Lemma \ref{lem:simplestable}] Throughout the proof of this lemma we will assume that $\N\gg 0$. There is a finite set $\mc{P}$ of $n$--partitions $\ll$ such that $\ll[\N]$ occurs in $V_{\N}$, and for each such $\ll$, the multiplicity of $\ll[\N]$ in $V_{\N}$ is $m_{\ll}$, independent of $\N$.

We fix now a size maximal $\ll\in\mc{P}$. It follows from Lemma \ref{lem:sizemaximal} that if $\N,\N'$ are consecutive $n$--tuples, then there are no nonzero $\S_{\N}$--equivariant maps between $(V_{\N})_{\ll[\N]}$ and $(V_{\N'})_{\tilde{\ll}[\N']}$ when $\tilde{\ll}\neq\ll$. Letting $W_{\N}=(V_{\N})_{\ll[\N]}$ for all $\N$ yields a (stable) submodule $W$ of $V$. We fix $\N^0\gg 0$ and define $U$ by letting $U_{\N^0}$ be a subrepresentation of $W_{\N^0}$ isomorphic to $\ll[\N^0]$, and letting $U_{\N}$ be the image of $U_{\N^0}$ via $\phi_{\N^0,\N}(\ll)$ when $\N\geq\N^0$ (and $U_{\N}=0$ otherwise). It is clear that $U$ is a submodule of $V$ equivalent to $V(\ll)$. 
\end{proof}

\begin{definition}[External tensor product of $\FI$--modules]\label{def:tensorproduct} For an $\FI^n$--module $V$ and an $\FI^m$--module $W$, we let $V\bt W$ denote their \defi{external tensor product}, defined by 
\[(V\bt W)_{N_1,\cdots,N_n,M_1,\cdots,M_m}=V_{N_1,\cdots,N_n}\o W_{M_1,\cdots,M_m},\]
with the natural induced maps. If $V$ and $W$ are (super)stable, then the same is true about $V\bt W$. Note that if $\ll$ is an $n$--partition, then the $\FI^n$--module $V(\ll)=V(\ll^1)\bt\cdots\bt V(\ll^n)$ is an external tensor product of $\FI^1$--modules.
\end{definition}

\begin{definition}[Convolution of $\FI$--modules]\label{def:convolution} Given two $\FI^n$--modules $V,W$, we define their \defi{convolution} $V*W$ by
\[(V*W)_{\N}=\bigoplus_{\N^1+\N^2=\N} \rm{Ind}_{\S_{\N^1}\times\S_{\N^2}}^{\S_{\N}}(V_{\N^1}\bt W_{\N^2}),\]
with the natural induced maps. In the functor notation, if $\ul{A}=(A_1,\cdots,A_n)$ denotes an $n$--tuple of sets, and if we write $\ul{A}=\ul{B}\sqcup\ul{C}$ to signify $A_i=B_i\sqcup C_i$ for all $i$, then
\[(V*W)_{\ul{A}}=\bigoplus_{\ul{B}\sqcup\ul{C}} V_{\ul{B}}\o W_{\ul{C}}.\]
\end{definition}
\noindent Note that tensor products and convolutions preserve exactness, and that they are associative.

Given an $n$--partition $\mu\vdash^n\ul{a}=(a_1,\cdots,a_n)$, we write $T(\mu)$ for the supertrivial $\FI^n$--module having $T(\mu)_{\ul{a}}=[\mu]$, and $T(\mu)_{\N}=0$ for all $\N\neq\ul{a}$. For general (super)stable modules $V,W$, it is not the case that $V*W$ is also stable. However, we will see in Theorem \ref{thm:convstable} below that convolution with modules of the form $T(\mu)$ (or more general supertrivial modules) preserves stability. If $V$ is any $\FI^n$--module then
\[(V*T(\mu))_{\N}=
\begin{cases}
\rm{Ind}_{\S_{\N-\ul{a}}\times\S_{\ul{a}}}^{\S_{\N}}(V_{\N-\ul{a}}\bt [\mu]) & \rm{if}\ \N\geq\ul{a} \\
0 & \rm{otherwise}
\end{cases}.
\]

When $V=V(\1)$ is the $\FI^n$--module corresponding to the empty partition ($V_{\N}=\1_{\S_{\N}}$ for all $\N$), $V*T(\mu)$ coincides with the multivariate analogue of the module $M(\mu)$ introduced in \cite{church-ellenberg-farb}. An important part of the theory of finitely generated $\FI$--modules that Church--Ellenberg--Farb develop is based on the fact that the modules $M(\mu)$ are finitely generated which is proved in \cite[Theorem~2.8]{church}. We formulate the following consequence/generalization of this theorem

\begin{theorem}[{\cite[Theorem~2.8]{church}}]\label{thm:convstable} If $V$ is a representation (super)stable $\FI^n$--module and $T$ is a supertrivial $\FI^n$--module, then the convolution $V*T$ is representation (super)stable. Moreover, if $V$ has stable range $\N'\geq\N$, and $\ul{a}=(a_1,\cdots,a_n)$ is such that $T_{\ul{a}'}=0$ for $\ul{a}'>\ul{a}$, then $V*T$ has stable range $\N'\geq \N+2\cdot\ul{a}$.
\end{theorem}

\begin{remark} In the language of \cite{sam-snowden}, the first part of the theorem says that the tensor product between a finitely generated module and a finite length module is finitely generated, which is a tautology in their context.
\end{remark}

\begin{proof}[Proof of Theorem \ref{thm:convstable}] As before, it is enough to treat the case when $V$ is stable, the superstable case being a direct consequence. Since $V$ is stable, it has a composition series by Remark \ref{rem:filtrationstable} with terms that are equivalent to $V(\ll)$. Since convolutions preserve exactness, it follows that we may assume $V=V(\ll)$ for some $\ll$. Similarly, since $T$ has a filtration with supertrivial modules of the form $T(\mu)$, we may assume that $T=T(\mu)$. Writing $V=V(\ll^1)\bt\cdots\bt V(\ll^n)$ as a tensor product of $\FI^1$--modules, and $\mu=(\mu^1,\cdots,\mu^n)$, it follows that
\[V*T(\mu)=(V(\ll^1)*T(\mu^1))\bt\cdots\bt(V(\ll^n)*T(\mu^n)).\]
To prove the stability of $V*T$ and the estimation for the stable range we're then reduced to the case when $n=1$, i.e. when $\ll$ and $\mu\vdash a$ are partitions. By the argument in Lemma \ref{lem:simplestable}, $V(\ll)$ is a submodule in $M(\ll)=V(\1)*T(\ll)$, hence $V(\ll)*T(\mu)$ is a submodule in $V(\1)*T(\ll)*T(\mu)=V(\1)*(T(\ll)*T(\mu))$ which is stable by \cite[Theorem~2.8]{church}. It follows from Lemma \ref{lem:subquotstable} that $V(\ll)*T(\mu)$ is also stable.

To end the proof of the theorem we need to show that if $V(\ll)$ has stable range $N'\geq N$ and $\mu\vdash a$, then $V(\ll)*T(\mu)$ has stable range $N'\geq N+2a$. Since $V(\ll)_m=0$ for $m\leq |\ll|+\ll_1$, we must have $N\geq |\ll|+\ll_1$. As noted before, $V(\ll)*T(\ll)$ is a submodule of $V(\1)*(T(\ll)*T(\mu))=\bigoplus_{\nu} V(\1)*(T(\nu)^{\oplus c_{\ll,\mu}^{\nu}})$ where $c_{\ll,\mu}^{\nu}$ are the Littlewood--Richardson coefficients. In particular all partitions $\nu$ that appear have $|\nu|=|\ll|+a$ and $\nu_1\leq\ll_1+\mu_1\leq\ll_1+a$. It follows that for such $\nu$, $V(\1)*T(\nu)$ has \defi{weight} \cite[Def.~2.50]{church-ellenberg-farb} at most $|\ll|+a$ and \defi{stability degree} \cite[Def.~2.34]{church-ellenberg-farb} at most $\ll_1+a$, and since $(|\ll|+a)+(\ll_1+a)\leq N+2a$ we get by \cite[Thm.~2.58]{church-ellenberg-farb} that $V(\1)*T(\nu)$ has stable range $N'\geq N+2a$. We conclude that the module $V(\1)*(T(\ll)*T(\mu))$ has stable range $N'\geq N+2a$ which by the last part of Lemma \ref{lem:trianglestable} implies that $V(\ll)*T(\mu)$ has injectivity range $N'\geq N+2a$. Observe now that by the Littlewood--Richardson rule the multiplicities of the irreducible representations $[\delta[N']]$ appearing in $(V(\ll)*T(\mu))_{N'}$ stabilize for $N'\geq N+2a$ which then implies that the stable range of $V(\ll)*T(\mu)$ is $N'\geq N+2a$.
\end{proof}

\section{Inductive approach to computing the homology of packing complexes}\label{sec:induction}

We fix a sequence $\d=(d_1,\cdots,d_n)$ of positive integers, and drop it from the notation for the rest of this section: we write $\C_{\ul{A}}$ for the packing complex $\C_{\ul{A}}^{\d}$ associated to the $n$--tuple of sets $\A=(A_1,\cdots,A_n)$ (Definition \ref{def:SVcomplexes}). We write $\C_{\a_1,\cdots,\a_n}$ for the full subcomplex of $\C_{\A}$ generated by the vertex ($0$--simplex) $\a=(\a_1,\cdots,\a_n)$ and all its adjacent vertices (also known as the \defi{star} of $\a$). If we write $A_i'=A_i\setminus\a_i$, and $\ul{A'}=(A_1',\cdots,A_n')$, then $\C_{\a_1,\cdots,\a_n}$ can be thought of as the cone over $\C_{\A'}$ ($\C_{\A'}$ is called the \defi{link} of $\a$). 

We now fix an $n$--tuple $\N=(N_1,\cdots,N_n)$ of positive integers and the corresponding complex $\C_{\N}$. We proceed to construct a long exact sequence that relates the reduced homology groups of $\C_{\N}$ to those of complexes $\C_{\N'}$, for $\N'\leq\N$. Such long exact sequences have been previously studied in the case of matching complexes by \cites{bouc,sha-wachs,jonsson}, and in that of chessboard complexes by \cites{blvz,sha-wachs}.

\begin{example}\label{ex:segre33} Assume that $n=2$, $N_1=N_2=3$ and $d_1=d_2=1$. Since the sets $\a_i$ are singletons, $\a_i=\{a_i\}$, we write $a_i$ instead of $\a_i$. If we take $a_1=a_2=3$, then the subcomplex $\C_{a_1,a_2}$ of $\C_{(3,3)}^{(1,1)}$ looks like
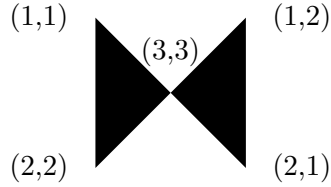
\[
\begin{tikzpicture}
[label distance = .5ex];
\fill[black] (0,0) node[label=above:{(3,3)}]{} -- (-1,-1) node[label=left:{(2,2)}]{} -- (-1,1) node[label=left:{(1,1)}]{} -- (0,0);
\fill[black] (0,0) -- (1,1) node[label=right:{(1,2)}]{} -- (1,-1) node[label=right:{(2,1)}]{} -- (0,0);
\end{tikzpicture}
\]
hence it is the cone over the complex $\C_{(2,2)}^{(1,1)}$ discussed in Example \ref{ex:c22}.
\end{example}

There is one situation when it is easy to compute the homology of $\C_{\N}^{\d}$, namely when it is zero--dimensional.

\begin{lemma}\label{lem:0dimension}
 Suppose that $N_j<2d_j$ for some $j=1,\cdots,n$. Then $\C_{\N}^{\d}$ is zero--dimensional (or empty), and
\[H_0(\C_{\N}^{\d})=\bigoplus_{\ll\vdash^n\N}[\ll],\]
where the sum is over $n$--partitions $\ll=(\ll^1,\cdots,\ll^n)$ with each $\ll^i$ having at most two parts and $\ll^i_1\geq\max(d_i,N_i-d_i)$.
\end{lemma}

\begin{proof}
 The first assertion follows from the definition of the complexes $\C_{\N}^{\d}$. To describe $H_0(\C_{\N}^{\d})$, note that it has a natural basis indexed by the $0$--simplices of $\C_{\N}^{\d}$. $\S_{\N}$ acts transitively on these simplices with stabilizers isomorphic to $\S_{\d}\times\S_{\N'}$, where $\N'=(N_1-d_1,\cdots,N_n-d_n)$. It follows that as a $\S_{\N}$--representation
\[H_0(\C_{\N}^{\d})=\rm{Ind}_{\S_{\d}\times\S_{\N'}}^{\S_{\N}}(\1_{\S_{\d}}\oo\1_{\S_{\N'}})\]
whose decomposition into irreducibles can then be computed using Pieri's rule.
\end{proof}

We shall assume from now on that $N_j\geq 2d_j$ for all $j$. Fix an index $i$ between $1$ and $n$, and an element $a_i\in A_i$. Consider the $n$--tuple $\ul{A^i}=(A_1,\cdots,A_i\setminus\{a_i\},\cdots,A_n)$. We have that $\C_{\ul{A^i}}$ is a subcomplex of $\C_{\ul{A}}$, hence we get a relative homology long exact sequence:
\begin{equation}\label{eq:les1}
 \cdots\lra \H_r(\C_{\ul{A^i}})\lra \H_r(\C_{\ul{A}})\lra H_r(\C_{\ul{A^i}},\C_{\ul{A}})\lra\cdots
\end{equation}
Note that this exact sequence is equivariant with respect to the action of $\S_{\ul{A^i}}\subset\S_{\ul{A}}$. We identify $H_r(\C_{\ul{A^i}},\C_{\ul{A}})$ with $\H_r(X^i)$, where $X^i$ is the quotient space $\C_{\ul{A}}/\C_{\ul{A^i}}$. We write $*$ for the image of $\C_{\ul{A^i}}$ in the quotient. $X^i$ is connected (because $N_j\geq 2d_j$ for all $j$), hence $\H_0(X^i)=0$, and furthermore, it is covered by subspaces $X^i_{\a_1,\cdots,\a_i,\cdots,\a_n}$, where $\a_j\subset A_j$ for all $j$, $a_i\in\a_i$, and
\[X^i_{\a_1,\cdots,\a_i,\cdots,\a_n}=\rm{Image}(\C_{\a_1,\cdots,\a_n}\subset\C^{\ul{A}}\lra X^i).\]
Since any $0$--simplex of $\C_{\a_1,\cdots,\a_n}$ distinct from $(\a_1,\cdots,\a_n)$ is contained in $\ul{A^i}$, it follows that any two distinct subspaces $X^i_{\a_1,\cdots,\a_i,\cdots,\a_n}$ of $X^i$ intersect in a single point, namely $*$. This shows that for $r>0$
\begin{equation}\label{eq:oplusXi}
 \H_r(X^i)=\bigoplus_{\substack{\a_j\subset A_j\\ a_i\in\a_i}}\H_r(X^i_{\a_1,\cdots,\a_i,\cdots,\a_n}).
\end{equation}
Note that $X^i_{\a_1,\cdots,\a_i,\cdots,\a_n}$ is obtained by taking the cone over $\C_{\ul{A'}}$ (where $A_j'=A_j\setminus\a_j$ for all $j$, as before), and then collapsing $\C_{\ul{A'}}$, so it can be naturally identified with the suspension of $\C_{\ul{A'}}$ (see Example \ref{ex:segre33smash} below). The effect of suspension on reduced homology is just a shift in degrees, thus
\[\H_r(X^i_{\a_1,\cdots,\a_i,\cdots,\a_n})=\H_{r-1}(\C_{\ul{A'}}).\]
Equation (\ref{eq:oplusXi}) then becomes
\begin{equation}\label{eq:HrXi}
\H_r(X^i)=\bigoplus_{\substack{\a_j\subset A_j\\ a_i\in\a_i}}\H_{r-1}(\C_{\ul{A'}(\a_1,\cdots,\a_n)}),
\end{equation}
where we write $\ul{A'}(\a_1,\cdots,\a_n)$ to emphasize the dependence of $\ul{A'}$ on the sets $\a_j$.

\begin{example}\label{ex:segre33smash} Continuing Example \ref{ex:segre33}, we fix the index $i=2$, and $a_2=3$. The quotient space $X^i$ is then
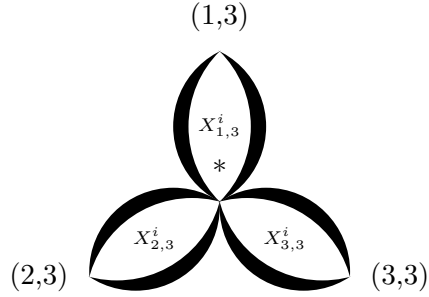
\[
\begin{tikzpicture}
 \def\firstcircle{(-0.5,1) circle (1.1180)}
 \def\secondcircle{(-1,1) circle (1.4142)}
 \begin{scope}[even odd rule]
  \clip \secondcircle (-2,-2) rectangle (2,3);
  \fill[black] \firstcircle;
 \end{scope}

 \def\firstcircle{(0.5,1) circle (1.1180)}
 \def\secondcircle{(1,1) circle (1.4142)}
 \begin{scope}[even odd rule]
  \clip \secondcircle (-2,-2) rectangle (3,3);
  \fill[black] \firstcircle;
 \end{scope}
 \def\firstcircle{(-0.6160,-0.9330) circle (1.1180)}
 \def\secondcircle{(-0.3660,-1.3660) circle (1.4142)}
 \begin{scope}[even odd rule]
  \clip \secondcircle (-3,-2.2) rectangle (1,1);
  \fill[black] \firstcircle;
 \end{scope}

 \def\firstcircle{(-1.1160,-0.0670) circle (1.1180)}
 \def\secondcircle{(-1.3660,0.3660) circle (1.4142)}
 \begin{scope}[even odd rule]
  \clip \secondcircle (-3,-2) rectangle (1,2);
  \fill[black] \firstcircle;
 \end{scope}
 \def\firstcircle{(0.6160,-0.9330) circle (1.1180)}
 \def\secondcircle{(0.3660,-1.3660) circle (1.4142)}
 \begin{scope}[even odd rule]
  \clip \secondcircle (-1,-2.2) rectangle (3,2);
  \fill[black] \firstcircle;
 \end{scope}

 \def\firstcircle{(1.1160,-0.0670) circle (1.1180)}
 \def\secondcircle{(1.3660,0.3660) circle (1.4142)}
 \begin{scope}[even odd rule]
  \clip \secondcircle (-1,-2) rectangle (3,2);
  \fill[black] \firstcircle;
 \end{scope}

\draw (0,0) node[label=above:{*}]{};
\draw (0,2) node[label=above:{(1,3)}]{};
\draw (-1.7320,-1) node[label=left:{(2,3)}]{};
\draw (1.7320,-1) node[label=right:{(3,3)}]{};
\draw (0,1) node{\tiny $X^i_{1,3}$};
\draw (-0.8660,-0.5) node{\tiny $X^i_{2,3}$};
\draw (0.8660,-0.5) node{\tiny $X^i_{3,3}$};

\end{tikzpicture}
\]
$X^i$ is covered by the three subsets $X^i_{j,3}$, $j=1,2,3$, each of which consists of two pairs of points, four $1$--cells and two $2$--cells. $X^i$ has a natural action of the product of symmetric groups $\S_3\times\S_2$. The subspace $X^i_{3,3}$ is the suspension of the complex in Example \ref{ex:c22}, whose only nonvanishing reduced homology group is $\H_0$ which is $1$--dimensional. It follows that $\H_1(X^i)$ has dimension $3$, which is not hard to see from the picture.
\end{example}

We can compute $\H_r(X^i)$ more precisely by keeping track of the equivariance of the decomposition (\ref{eq:HrXi}) with respect to the group $\S_{\ul{A^i}}=\S_{A_1}\times\cdots\times\S_{A_i\setminus\{a_i\}}\times\cdots\times\S_{A_n}$. Let us fix a collection $\a_1,\cdots,\a_n$, with $a_i\in\a_i$, and the corresponding $n$--tuple $\ul{A'}$. For $j=1,\cdots,n$, $j\neq i$, we have a natural inclusion of $\S_{A_j'}\times\S_{\a_j}\subset\S_{A_j}$. Similarly, we have $\S_{A_i'}\times \S_{\a_i\setminus\{a_i\}}\subset\S_{A_i\setminus\{a_i\}}$. Denoting by $\S_{\ul{\a^i}}$ the product $\S_{\a_1}\times\cdots\times\S_{\a_i\setminus\{a_i\}}\times\cdots\times\S_{\a_n}$, the previous inclusions give rise to a natural containment
\[H=\S_{\ul{A'}}\times\S_{\ul{\a^i}}\subset\S_{\ul{A^i}}.\]
The space $X^i_{\a_1,\cdots,\a_n}$ admits a natural action of the group $H$, where the factor $\S_{\ul{\a^i}}$ acts trivially. The reduced homology groups $\H_r(X^i_{\a_1,\cdots,\a_n})$ are therefore $H$--representations. The complex $\C_{\ul{A'}}$ has a natural $\S_{\ul{A'}}$--action. We can extend this to an $H$--action by letting $\S_{\ul{\a^i}}$ act trivially. It follows that the identification
\[\H_r(X^i_{\a_1,\cdots,\a_i,\cdots,\a_n})=\H_{r-1}(\C_{\ul{A'}})\]
is in fact an equality of $H$--modules. Moreover, if we write $\mc{S}$ for a system of representatives of the collection of left cosets $\S_{\ul{A^i}}/H$, then we can rewrite the decomposition (\ref{eq:HrXi}) as
\[\H_r(X^i) = \bigoplus_{\s\in\mc{S}} \s \cdot \H_{r-1}(\C_{\ul{A'}}),\]
or alternatively
\[\H_r(X^i) = \rm{Ind}_H^{\S_{\ul{A^i}}}\left(\H_{r-1}(\C_{\ul{A'}})\right).\]

Putting everything together, we obtain the following

\begin{proposition}\label{prop:les}
 Fix a sequence $\d=(d_1,\cdots,d_n)$ of positive integers. Consider sets $A_1,\cdots,A_n$, with $|A_j|=N_j\geq 2d_j$ for $j=1,\cdots,n$. Fix an index $i\leq n$ and an element $a_i\in A_i$. Let $\a_1,\cdots,\a_n$ be subsets of $A_1,\cdots,A_n$ respectively, with $a_i\in\a_i$. Let $A_j'=A_j\setminus\a_j$, for $j=1,\cdots,n$, and write $\ul{A}=(A_1,\cdots,A_n)$, $\ul{A'}=(A_1',\cdots,A_n')$, $\ul{A^i}=(A_1,\cdots,A_i\setminus\{a_i\},\cdots,A_n)$ and $\ul{\a^i}=(\a_1,\cdots,\a_i\setminus\{a_i\},\cdots,\a_n)$. We have a long exact sequence
\[
\begin{split}
\cdots\to &\rm{Ind}_{\S_{\ul{A'}}\times\S_{\ul{\a^i}}}^{\S_{\ul{A^i}}}\left(\H_{r}(\C_{\ul{A'}}^{\d})\o \1\right)\to \H_r(\C_{\ul{A^i}}^{\d})\to \rm{Res}_{\S_{\ul{A^i}}}^{\S_{\ul{A}}}(\H_r(\C_{\ul{A}}^{\d}))\to \\
\to &\rm{Ind}_{\S_{\ul{A'}}\times\S_{\ul{\a^i}}}^{\S_{\ul{A^i}}}\left(\H_{r-1}(\C_{\ul{A'}}^{\d})\o \1\right)\to\H_{r-1}(\C_{\ul{A^i}}^{\d})\to \rm{Res}_{\S_{\ul{A^i}}}^{\S_{\ul{A}}}(\H_{r-1}(\C_{\ul{A}}^{\d}))\to\cdots
\end{split}
\]
which is equivariant with respect to the action of the group $\S_{\ul{A^i}}$.
\end{proposition}

\begin{remark}
 If we make the convention that $\H_{-1}(\C_{\ul{A'}}^{\d})$ is the trivial $\S_{\ul{A'}}$--representation when $\C_{\ul{A'}}^{\d}$ is empty (i.e. $N_j<2d_j$ for some $j$), then the conclusion of the proposition remains true when $N_j$ is allowed to be smaller than $2d_j$.
\end{remark}


\begin{example}\label{ex:segre33homology} We continue with Example \ref{ex:segre33smash}. Note that the only nonzero reduced homology group of $\C_{(2,2)}^{(1,1)}$ is $\H_0$, and as explained in the introduction, its description as a $\S_2\times\S_2$--module is
\[\Yvcentermath1\H_0(\C_{(2,2)}^{(1,1)})={\tiny\yng(1,1)\o\yng(1,1)}.\]
Inducing up to $\S_3\times\S_2$ and using Pieri's rule, we obtain
\[\rm{Ind}_{\S_2\times\S_2}^{\S_3\times\S_2}(\H_0(\C_{(2,2)}^{(1,1)})) = \Yvcentermath1\left({\tiny\yng(2,1)+\yng(1,1,1)}\right)\o{\tiny\yng(1,1)}.\]
Using the arguments we're about to present (we leave this as an exercise for the interested reader), one can deduce that the only nonzero reduced homology group of $\C_{(3,2)}^{(1,1)}$ is $\H_1$, and
\[\Yvcentermath1\H_1(\C_{(3,2)}^{(1,1)})={\tiny\yng(1,1,1)\o\yng(2)}.\]
We would like to compute the reduced homology groups of $\C_{(3,3)}^{(1,1)}$. The long exact sequence in Proposition \ref{prop:les} yields
\[0\to\H_1(\C_{(3,2)}^{(1,1)})\to\H_1(\C_{(3,3)}^{(1,1)})\to \rm{Ind}_{\S_2\times\S_2}^{\S_3\times\S_2}(\H_0(\C_{(2,2)}^{(1,1)}))\to\H_0(\C_{(3,2)}^{(1,1)})\to\H_0(\C_{(3,3)}^{(1,1)})\to 0, \]
i.e.
\[\Yvcentermath1 0\to{\tiny\yng(1,1,1)\o\yng(2)}\to\H_1(\C_{(3,3)}^{(1,1)})\to\left({\tiny\yng(2,1)+\yng(1,1,1)}\right)\o{\tiny\yng(1,1)}\to 0\to\H_0(\C_{(3,3)}^{(1,1)})\to 0.\]
This forces $\H_0(\C_{(3,3)}^{(1,1)})=0$, which can also be seen from the fact that $\C_{(3,3)}^{(1,1)}$ is connected, and moreover
\begin{equation}\label{eq:resH1C33}
\Yvcentermath1 \rm{Res}_{\S_3\times\S_2}^{\S_3\times\S_3}\left(\H_1(\C_{(3,3)}^{(1,1)})\right)={\tiny\yng(1,1,1)\o\yng(2)+\yng(2,1)\o\yng(1,1)+\yng(1,1,1)\o\yng(1,1)}.
\end{equation}
There are two irreducible $\S_3\times\S_3$--representations whose restrictions to $\S_3\times\S_2$ contain the representation $[(1,1,1)]\o[(2)]$, namely
\[\Yvcentermath1 {\tiny\yng(1,1,1)\o\yng(3)\quad \rm{and}\quad \yng(1,1,1)\o\yng(2,1)}.\]
If $[(1,1,1)]\o[3]$ has positive multiplicity in $\H_1(\C_{(3,3)}^{(1,1)})$, then by symmetry the same is true for $[(3)]\o[(1,1,1)]$. But then
\[\Yvcentermath1\tiny\rm{Res}_{\S_3\times\S_2}^{\S_3\times\S_3}\left(\yng(3)\o\yng(1,1,1)\right)=\yng(3)\o\yng(1,1)\]
would have positive multiplicity inside  $\rm{Res}_{\S_3\times\S_2}^{\S_3\times\S_3}\left(\H_1(\C_{(3,3)}^{(1,1)})\right)$, which is not the case. It follows that $[(1,1,1)]\o[(2,1)]$ must occur in $\H_1(\C_{(3,3)}^{(1,1)})$, and by symmetry the same has to be true about $[(2,1)]\o[(1,1,1)]$. Since
\[\Yvcentermath1\tiny \rm{Res}_{\S_3\times\S_2}^{\S_3\times\S_3}\left(\yng(1,1,1)\o\yng(2,1)+\yng(2,1)\o\yng(1,1,1)\right)=\yng(1,1,1)\o\yng(2)+\yng(2,1)\o\yng(1,1)+\yng(1,1,1)\o\yng(1,1)\]
coincides with the restriction of $\H_1(\C_{(3,3)}^{(1,1)})$ to $\S_3\times\S_2$ (see \ref{eq:resH1C33}), this forces
\[\Yvcentermath1\tiny \H_1(\C_{(3,3)}^{(1,1)})=\yng(1,1,1)\o\yng(2,1)+\yng(2,1)\o\yng(1,1,1).\]
Note that this coincides with the description of the functor $K_{2,1}$ in Figure \ref{fig:2factorsyz} on page \pageref{fig:2factorsyz}. That this should be the case is a consequence of Theorem \ref{thm:main}.
\end{example}

\section{Representation stability for packing complexes}\label{sec:SVstability}

In this section we prove the stabilization of the homology groups of packing complexes. The argument is based on the general results on representation stability established in Section \ref{sec:stability}.

\begin{theorem}\label{thm:stability} We fix $n>0$ and an $n$--tuple of positive integers $\d=(d_1,\cdots,d_n)$. For $k\geq -1$ the $\FI^n$--module $\mc{H}_k$ defined by letting
\[(\mc{H}_k)_{\A}=\H_k\left(\C_{\A}^{\d}\right)\]
whenever $\A$ is an $n$--tuple of finite sets is representation superstable and trivial.

Moreover, if $r<n$ and if we fix $(n-r)$ sets, say $A_{r+1},\cdots,A_n$, of cardinalities $N_{r+1},\cdots,N_n$ respectively, and if we let 
\[m=\min_{j=r+1,\cdots,n}\left\lfloor\frac{N_j}{d_j}\right\rfloor,\]
then the pull--back $\FI^r$--module $\mc{H}_k(N_{r+1},\cdots,N_n)$ defined by
\[\mc{H}_k(N_{r+1},\cdots,N_n)_{(A_1,\cdots,A_r)}=(\mc{H}_k)_{(A_1,\cdots,A_n)}\]
has stable range $\N'\geq\N$, where $\N=2m\cdot(d_1,d_2,\cdots,d_r)$.
\end{theorem}

\begin{corollary}\label{cor:stability} For $k\geq -1$ and fixed values of the parameters $N_{r+1},\cdots,N_n$, there exist a finite number of $n$--partitions $\ll=(\ll^1,\cdots,\ll^n)$ and multiplicities $m_{\ll}>0$ such that the decomposition
\[\H_k\left(\C_{(N_1,\cdots,N_n)}^{(d_1,\cdots,d_n)}\right)=\bigoplus_{\ll} ([\ll^1[N_1]]\o\cdots\o[\ll^r[N_r]]\o[\ll^{r+1}]\o\cdots\o[\ll^n])^{\oplus m_{\ll}}\]
holds for $N_i\geq 2m\cdot d_i$, $i=1,\cdots,r$, where $m=\min\{\lfloor N_j/d_j\rfloor:j=r+1,\cdots,n\}$.
\end{corollary}

\begin{theorem}\label{cor:syzstability} Consider $r<n$, a sequence $\d=(d_1,\cdots,d_n)$ of positive integers, and fix nonnegative integers $p,q$ and $b_{r+1},\cdots,b_n$ such that the inequality $b_j<d_j$ holds for at least one value of $j\in\{r+1,\cdots,n\}$. For integers $b_1,\cdots,b_r$ we let $N_i=(p+q)d_i+b_i$. There exist a finite number of $n$--partitions $\ll$ and corresponding multiplicities $m_{\ll}$ such that the decomposition
\[K_{p,q}^{\d}(b_1,\cdots,b_r,b_{r+1},\cdots,b_n)=\bigoplus_{\ll} \left(S_{\ll^1[N_1]}\oo\cdots\oo S_{\ll^r[N_r]}\oo S_{\ll^{r+1}}\oo\cdots\oo S_{\ll^n}\right)^{\oplus m_{\ll}}\]
holds independently of $b_1,\cdots,b_r$ as long as $b_i\geq (p+q)d_i$, $i=1,\cdots,r$.
\end{theorem}

\begin{proof} The result follows from Corollary~\ref{cor:stability} and from Theorem \ref{thm:main}, which describes the relationship between $K_{p,q}^{\d}(\ul{b})$ and $\H_{p-1}(\C_{\N}^{\d})$, where $N_j=(p+q)d_j+b_j$. With the notation in Corollary~\ref{cor:stability} we have
\[m=\min_{j=r+1,\cdots,n}\left\lfloor\frac{N_j}{d_j}\right\rfloor=p+q,\]
since by assumption $0\leq b_j<d_j$ for at least one value of $j\in\{r+1,\cdots,n\}$. The conclusion now follows by observing that the condition $N_i\geq 2m d_i=2(p+q)d_i$ of Corollary \ref{cor:stability} is equivalent to $b_i\geq (p+q)d_i$, $i=1,\cdots,r$. 
\end{proof}

\begin{proof}[Proof of Theorem \ref{thm:stability}] The fact that the functors $\mc{H}_k$ are stable and trivial follows from Theorem \ref{thm:vanishing}. Assume now that $r<n$. By Definition \ref{def:superstability}, to prove superstability we need to show that fixing any $(n-r)$ of the parameters $N_1,\cdots,N_n$ (for simplicity of notation we will assume that they are $N_{r+1},\cdots,N_n$), the corresponding pull--back $\FI^r$--module $\mc{H}_k(N_{r+1},\cdots,N_n)$ is stable. We prove this statement by induction on the $(n-r)$--tuple $(N_{r+1},\cdots,N_n)$, considering the lexicographical ordering of tuples. Note that if $\N=(N_1,\cdots,N_r)$, then $\mc{H}_k(N_{r+1},\cdots,N_n)$ is given by
\[\mc{H}_k(N_{r+1},\cdots,N_n)_{\N} = \H_k\left(\C_{(N_1,\cdots,N_n)}^{(d_1,\cdots,d_n)}\right).\]

If $N_i<d_i$ for some $i=r+1,\cdots,n$, then $\C_{(N_1,\cdots,N_n)}^{(d_1,\cdots,d_n)}$ is empty, so the only nonzero module $\mc{H}_k$ is $\mc{H}_{-1}=V(\1)\o\1_{\S_{(N_{r+1},\cdots,N_n)}}$, where $V(\1)$ is the stable module corresponding to the empty $r$--partition, i.e. $V(\1)_{\N}$ is the trivial $\S_{\N}$--representation for every $\N$. Note that 
\[m=\min\{\lfloor N_j/d_j\rfloor:j=r+1,\cdots,n\}=0\]
in this case and that $V(\1)$ has stable range $\N'\geq\ul{0}=(0,\cdots,0)$, so the estimation of the stable range holds.

If $N_i\geq d_i$ for all $i$, and $N_j<2d_j$ for some $j$, then the only nonzero $\mc{H}_k$ is $\mc{H}_0$, and it follows from Lemma \ref{lem:0dimension} (which computes $H_0$ rather than $\H_0$) that we have an exact sequence
\[0\to\mc{H}_0\to (V(\1)*T(\mu))\o\rm{Ind}_{\S_{(d_{r+1},\cdots,d_n)}}^{\S_{(N_{r+1},\cdots,N_n)}}(\1_{\S_{(d_{r+1},\cdots,d_n)}})\to V(\1)\o\1_{\S_{(N_{r+1},\cdots,N_n)}}\to 0,\]
where $\mu=((d_1),(d_2),\cdots,(d_r))$ is the $r$--partition corresponding to the trivial $\S_{(d_1,\cdots,d_r)}$--representation. Since $V(\1)$ has stable range $\N'\geq\ul{0}$, and since $V(\1)*T(\mu)$ is stable with stable range $\N'\geq 2\cdot(d_1,\cdots,d_r)$ (Theorem \ref{thm:convstable}), it follows from Lemma \ref{lem:subquotstable} that $\mc{H}_0$ is also stable with stable range $\N'\geq 2\cdot(d_1,\cdots,d_r)$. Note that 
\[m=\min\{\lfloor N_j/d_j\rfloor:j=r+1,\cdots,n\}=1\]
so the estimation of the stable range holds in this case as well.

We may then assume that $N_i\geq 2d_i$ for all $i=r+1,\cdots,n$. Applying Proposition \ref{prop:les} with $i=n$ and $a_i=N_n$, we get an exact triangle $X_{\bullet}\to Y_{\bullet}\to Z_{\bullet}\to X_{\bullet}[-1]$, where
\[Y_k =\mc{H}_k(N_{r+1},\cdots,N_n-1),\quad Z_k=\mc{H}_k(N_{r+1},\cdots,N_n),\]
and $X_k$ is a direct sum of copies of \[T(\mu)*\mc{H}_{k}(N_{r+1}-d_{r+1},\cdots,N_n-d_n),\] where $\mu=((d_1),(d_2),\cdots,(d_r))$. More precisely, for each $\N=(N_1,\cdots,N_r)$ we have that $\left(T(\mu)*\mc{H}_{k}(N_{r+1}-d_{r+1},\cdots,N_n-d_n)\right)_{\N}$ is a $\S_{(N_{r+1}-d_{r+1},\cdots,N_n-d_n)}$--representation and
\[(X_k)_{\N}=\rm{Ind}_{\S_{(N_{r+1}-d_{r+1},\cdots,N_n-d_n)}\times\S_{(d_{r+1},\cdots,d_n-1)}}^{\S_{(N_{r+1},\cdots,N_n-1)}}\left(\left(T(\mu)*\mc{H}_{k}(N_{r+1}-d_{r+1},\cdots,N_n-d_n)\right)_{\N}\boxtimes\1\right),\]
where the $\1$ on the RHS denotes the trivial $\S_{(d_{r+1},\cdots,d_n-1)}$--representation. By induction the $Y_k$'s are stable with stable range $\N'\geq 2m\cdot(d_1,\cdots,d_r)$ and the $\mc{H}_{k}(N_{r+1}-d_{r+1},\cdots,N_n-d_n)$'s are stable with stable range $\N'\geq 2(m-1)\cdot(d_1,\cdots,d_r)$. Since $\mu\vdash^r(d_1,\cdots,d_r)$, it follows from the last part of Theorem \ref{thm:convstable} that the $X_k$'s are stable with stable range 
\[\N'\geq 2\cdot(d_1,\cdots,d_r)+2(m-1)\cdot(d_1,\cdots,d_r)=2m\cdot(d_1,\cdots,d_r).\]
We can now apply Lemma \ref{lem:trianglestable} to conclude that the $Z_k$'s are also stable with stable range $\N'\geq 2m\cdot(d_1,\cdots,d_r)$, concluding the proof of the theorem.
\end{proof}

\section{An example: the linear strand}\label{sec:examples}

In this section we show that for certain line bundles on Segre varieties, the decomposition into irreducible representations of the linear syzygy modules is as hard to compute as the decomposition of the plethysms $\bigwedge^p(V_1\oo\cdots\oo V_n)$. This gives an indication of how difficult the problem of computing syzygies for line bundles on Segre--Veronese varieties should be.

We write $K_{p,0}(a)$ for the syzygy functor $K_{p,0}^{\d}(a,0,\cdots,0)$ where $\d=(1,\cdots,1)$ (as defined in Section \ref{subsec:syzfun}). It describes the linear syzygies for the bundle $\mc{B}_{a}=\mc{O}(a,0,\cdots,0)$ with respect to the Segre embedding corresponding to $\mc{L}=\mc{O}(1,\cdots,1)$. We have the following

\begin{theorem}\label{thm:linearstrand} For $n\geq 2$, $p\geq 0$ and $\ll=(\ll^1,\cdots,\ll^n)$ a collection of partitions of $p$ we let $m_{\ll}$ denote the multiplicity of $S_{\ll^1}V_1\oo\cdots\oo S_{\ll^n}V_n$ inside $\bigwedge^p(V_1\oo\cdots\oo V_n)$. We have
\[K_{p,0}(a)=\bigoplus_{\ll\vdash^n (p,\cdots,p)} (S_{\ll^1[p+a]}\oo S_{\ll^2}\oo\cdots\oo S_{\ll^n})^{\oplus m_{\ll}},\]
where the functor $S_{\ll^1[p+a]}$ is identically zero when $\ll^1_1>a$.
\end{theorem}

\begin{remark} The sequence $K_{p,0}(a)$ stabilizes (in the sense of Section \ref{sec:stability}) for $a\geq p$. 
\end{remark}

\begin{proof}[Proof of Theorem \ref{thm:linearstrand}] The proof of Theorem \ref{thm:linearstrand} is based on the techniques from \cite{friedman-hanlon}. Note that by Theorem \ref{thm:main} it suffices to show that
\[\H_{p-1}(\C_{p+a,p,\cdots,p})=\bigoplus_{\ll\vdash^n (p,\cdots,p)} ([\ll^1[p+a]]\oo[\ll^2]\oo\cdots\oo [\ll^n])^{\oplus m_{\ll}}\]
for all $p\geq 0$. As in the proof of Theorem \ref{thm:main}, $\H_{p-1}(\C_{p+a,p,\cdots,p})$ can be computed as the kernel of the map $\partial:D_p\to D_{p-1}$, where $D_p$ is a vector space with a basis consisting of elements
\[z_{\a^1}\wedge\cdots\wedge z_{\a^p},\]
where $\a^i=(a^i_1,\cdots,a^i_n)$, for $a^i_1\in A_1=\{1,\cdots,p+a\}$, $a^i_j\in A_j=\{1,\cdots,p\}$ for $j>1$, with $a^i_j\neq a^{i'}_j$ for $i\neq i'$, and
\[\partial\left(z_{\a^1}\wedge\cdots\wedge z_{\a^p}\right)=\sum_{i=1}^p (-1)^{i-1} z_{\a^1}\wedge\cdots\wedge\widehat{z_{\a^i}}\wedge\cdots\wedge z_{\a^p}.\]
Consider the \defi{transpose} operator $\partial^*:D_{p-1}\to D_p$, defined by
\[\partial^*\left(z_{\a^1}\wedge\cdots\wedge z_{\a^{p-1}}\right)=\sum_{\beta} z_{\beta}\wedge z_{\a^1}\wedge\cdots\wedge z_{\a^{p-1}},\]
where the sum ranges over $n$--tuples $\b=(b_1,\cdots,b_n)$ with $b_j\neq a^i_j$ for all $i,j$. Note that $b_j$ is uniquely determined for $j=2,\cdots,n$, since $|A_j|=p$. Let $\Delta=\partial^*\circ\partial$ denote the \defi{Laplacian} operator. By \cite[Prop.~1]{friedman-hanlon} the kernel of $\Delta$ (the set of \defi{harmonic $p$--forms}) coincides with the kernel of $\partial$, so it suffices to understand the decomposition into irreducible $\S_{\ul{A}}\simeq\S_{p+a}\times\S_p\times\cdots\times\S_p$--representations of the $0$--eigenspace of $\Delta$.

We now consider the spaces $C_p,C_{p-1}$ defined in analogy with $D_p,D_{p-1}$, replacing $\wedge$ by $\oo$. More precisely, $C_p$ has a basis
\[z_{\a}=z_{\a^1}\oo\cdots\oo z_{\a^p},\]
where $\a^i=(a^i_1,\cdots,a^i_n)$, for $a^i_j\in A_j$, with $a^i_j\neq a^{i'}_j$ for $i\neq i'$. We can identify $z_{\a}$ with a $p\times n$ table whose $(i,j)$--entry is $a^i_j$. Besides the left action of $\S_{\ul{A}}$ that permutes the elements of the sets $A_1,\cdots,A_n$, $C_p$ admits a right action (which we denote by the symbol $*$) of $\S_p^n$, where the $j$--th factor acts by permuting the $j$--th column of a table. We identify $\S_p^n$ with $\S_{\ul{B}}=\S_{B_1}\times\cdots\times\S_{B_n}$, where $B_j$ is the set of boxes in the $j$--th column of a table.	

\begin{example}\label{ex:leftrightaction} Let $n=4$, $p=3$ and $a=2$. Consider the element $z_{\a}=z_{(2,1,2,3)}\oo z_{(4,3,1,1)}\oo z_{(3,2,3,2)}\in C_p$ corresponding to the table
\[
M=\begin{array}{|c|c|c|c|}
 \hline
2 & 1 & 2 & 3\\ \hline
4 & 3 & 1 & 1\\ \hline
3 & 2 & 3 & 2\\ \hline
\end{array}
\]
Thinking of the transposition $(1,2)$ first as an element of $\S_{A_1}$ and then as one of $\S_{B_1}$ we get
\[(1,2)\cdot M = \begin{array}{|c|c|c|c|}
 \hline
1 & 1 & 2 & 3\\ \hline
4 & 3 & 1 & 1\\ \hline
3 & 2 & 3 & 2\\ \hline
\end{array},\quad M * (1,2) = \begin{array}{|c|c|c|c|}
 \hline
4 & 1 & 2 & 3\\ \hline
2 & 3 & 1 & 1\\ \hline
3 & 2 & 3 & 2\\ \hline
\end{array}.\]
The action of $(1,2)\in\S_{A_3}$ on $z_{\a}$ coincides with that of $(1,2)\in\S_{B_3}$, both yielding the element $z_{(2,1,1,3)}\oo z_{(4,3,2,1)}\oo z_{(3,2,3,2)}\in C_p$, but this is not the case for $(1,2)\in\S_{A_4}$ and $(1,2)\in\S_{B_4}$:
\[(1,2)\cdot M = \begin{array}{|c|c|c|c|}
 \hline
2 & 1 & 2 & 3\\ \hline
4 & 3 & 1 & 2\\ \hline
3 & 2 & 3 & 1\\ \hline
\end{array},\quad M * (1,2)= \begin{array}{|c|c|c|c|}
 \hline
2 & 1 & 2 & 1\\ \hline
4 & 3 & 1 & 3\\ \hline
3 & 2 & 3 & 2\\ \hline
\end{array}.\] 
\end{example}

The actions of $\S_{\ul{A}}$ and $\S_{\ul{B}}$ commute, so the vector space $C_p$ is a representation of $\S=(\S_{A_1}\times\S_{B_1})\times\cdots\times(\S_{A_n}\times\S_{B_n})$. Moreover, we have $D_p=C_p * c$ for
\[c=\sum_{\s\in\S_p}\rm{sgn}(\s)\cdot\s,\]
where we think of $\S_p$ as the diagonal subgroup of $\S_{\ul{B}}$ of permutations of the rows of the tables in $C_p$. By \cite[Thm.~3]{friedman-hanlon}, we have the decomposition into irreducible $\S$--representations
\[C_p\simeq\bigoplus_{\substack{\ll\vdash^n(p+a,p,\cdots,p)\\ \mu\vdash^n(p,p,\cdots,p)}} ([\ll^1]\oo[\mu^1])\oo([\ll^2]\oo[\mu^2])\oo\cdots\oo([\ll^n]\oo[\mu^n]),\]
where $\ll,\mu$ vary over all $n$--partitions with the property that $\ll^i=\mu^i$ when $i>1$, and $\ll^1$ is obtained from $\mu^1$ by adding $a$ boxes, no two in the same column. We write $C(\ll,\mu)$ for the summand in the decomposition of $C_p$ corresponding to a given pair $(\ll,\mu)$ of $n$--partitions.

We define the operator $T:C_p\to C_p$ (see also the definition of the map $D_{r,n}$ on \cite[p.197]{friedman-hanlon}) by
\[T= \left(\sum_{i<j\in A_1}(i,j)\right) - \left(\sum_{i<j\in B_1}(i,j)\right) + \left(p - {a \choose 2}\right)\rm{Id},\]
where $(i,j)$ denote transpositions in $\S_{A_1}$ or $\S_{B_1}$. Note that $T$ commutes with right multiplication by $c$, and the induced map $T*c:C_p*c\to C_p*c$ coincides with the Laplacian $\Delta:D_p\to D_p$. By \cite[Lemma~1]{friedman-hanlon}, $T$ acts on $C(\ll,\mu)$ by multiplication by
\[C_{\ll^1}-C_{\mu^1}+p-{a \choose 2},\]
where for a partition $\delta$, the \defi{content} $C_{\delta}$ of $\delta$ is defined as the sum of the horizontal coordinates of the boxes of the associated Young diagram minus the sum of the vertical coordinates. For example in the case of the partition $\delta=(6,3,3,1)$, $C_{\delta}=9$ is the sum of the entries in the tableau
\[\Yvcentermath1\Yboxdim{15pt}\young(012345,\mone 01,\mtwo\mone 0,\mthree).\]
Now since $\ll^1$ is obtained from $\mu^1$ by adding $a$ boxes, no two in the same column, we get that
\[C_{\ll^1}-C_{\mu^1}=\sum_{\substack{j\\ (\ll^1)'_j=(\mu^1)'_j + 1}} \left(j-1-(\mu^1)'_j\right)\geq\sum_{j=1}^a \left(j-1-(\mu^1)'_j\right)\geq{a\choose 2}-p,\]
with equality if and only if $\mu^1_1\leq a$ and $\ll^1$ is obtained from $\mu^1$ by adding a row of length $a$, i.e. $\ll^1=\mu^1[a+p]$. We get that $C(\ll,\mu)$, which lies in the $\left(C_{\ll^1}-C_{\mu^1}+p-{a \choose 2}\right)$--eigenspace of $T$, is a kernel element precisely when the condition $\ll^1=\mu^1[a+p]$ is satisfied. The conclusion of the theorem now follows from the fact that the dimension of the vector space
\[([\mu^1]\oo[\mu^2]\oo\cdots\oo[\mu^n]) * c\]
coincides with the multiplicity $m_{\ll}$ of $S_{\mu^1}V_1\oo\cdots\oo S_{\mu^n}V_n$ inside $\bigwedge^p(V_1\oo\cdots\oo V_n)$ by Schur--Weyl duality.
\end{proof}

Similar techniques can be used to obtain a description of the linear syzygies of the line bundle $\mc{B}=\mc{O}(1)$ on $\bb{P}V$ with respect to the Veronese embedding corresponding to $\mc{L}=\mc{O}(d)$. We leave it as an exercise for the interested reader to prove the following

\begin{theorem}\label{thm:linearstrandVero}
 For $p\geq 0$, $d>0$, and $\ll$ a partition of $p\cdot(d-1)$, we let $m_{\ll}$ denote the multiplicity of $S_{\ll}V$ inside $\Sym^p(\Sym^{d-1} V)$. We have
\[K_{p,0}^d(1)=\bigoplus_{\ll\vdash p\cdot(d-1)} S_{\tilde{\ll}}^{\oplus m_{\ll}},\]
where $\tilde{\ll}$ is obtained from $\ll$ by adding one column of height $(p+1)$ to the beginning of its Young diagram, i.e. if $\ll=(\ll_1,\cdots,\ll_p)$ with $\ll_i\geq 0$, then $\tilde{\ll}=(1+\ll_1,1+\ll_2,\cdots,1+\ll_p,1)$.
\end{theorem}

\noindent Note that by a result of Newell \cite{newell} the multiplicity $m_{\ll}$ in the above decomposition coincides with that of $S_{\overline{\ll}}V$ inside the plethysm $\bigwedge^p(\Sym^d V)$, where $\overline{\ll}=(1+\ll_1,\cdots,1+\ll_p)$ is a partition of $pd$. The above theorem gives a concrete description of the syzygy functors $K_{p,0}^d(1)$ which fits in with the more general theory of \cite{fulger-zhou} that gives a quantitative measure of the asymptotic complexity of the functors $K_{p,0}^d(b)$ and $K_{p,1}^d(b)$ as $d$ becomes very large. 

\section*{Appendix: asymptotic vanishing of syzygies}

\setcounter{theorem}{0}
\renewcommand{\thetheorem}{A\arabic{theorem}}

In this appendix we explain how Ein and Lazarsfeld's notion of asymptotic vanishing for syzygies of arbitrary varieties \cite[Conjecture~7.1]{ein-laz} reduces to an asymptotic vanishing statement for line bundles on projective space (or on a product of projective spaces). The advantage of this reduction is that it transforms the problem of proving asymptotic syzygy vanishing into a very concrete one that admits numerous reformulations, situating it at the confluence of algebraic geometry, representation theory and combinatorial topology.

For $q\geq 2$ and $\ul{b}\in\bb{Z}^n$ let $P_{q,\ul{b}}(\d)$ be functions with the property that the syzygy functors $K_{p,q}^{\d}(\ul{b})$ (defined in Section \ref{subsec:syzfun}) vanish identically for $p\leq P_{q,\ul{b}}(\d)$. When $q=2$, we can take $P_{2,\ul{b}}(\d)=\min\{d_i+b_i:i=1,\cdots,n\}$ (Corollary \ref{cor:Np}). In the case $n=1$, Ein and Lazarsfeld conjectured that we can take $P_{q,b}(d)$ to be a polynomial of degree $(q-1)$ in $d$ \cite[Conjecture~7.6]{ein-laz}. We won't attempt to make a conjecture for what the best $P_{q,\ul{b}}(\d)$ would be when $n>1$, but a first naive guess that the reader might want to keep in mind for the discussion to follow would be to take $P_{q,\ul{b}}(\d)=\min\{P_i(d_i):i=1,\cdots,n\}$, for some polynomials $P_i$ of degree $(q-1)$. This guess is supported by the fact that if \cite[Conjecture~7.1]{ein-laz} were true, and $d_i=u_i\cdot d+v_i$ were linear functions of some parameter $d$ with $u_i>0$, then $P(d)=P_{q,\ul{b}}(\d)$ would have to grow as a polynomial of degree $(q-1)$ in $d$ (in the statement of the conjecture take $X=\bb{P}V_1\times\cdots\times\bb{P}V_n$, $B=\mc{O}(\ul{b})$, $A=\mc{O}(u_1,\cdots,u_n)$, $P=\mc{O}(v_1,\cdots,v_n)$). Our goal is to show that, regardless of their description, the functions $P_{q,\ul{b}}(\d)$ control the vanishing of syzygies of arbitrary modules, as explained below.

Given finite dimensional $\K$--vector spaces $V_1,\cdots,V_n$ we write $V=V_1\oplus\cdots\oplus V_n$ and $S=\Sym(V)$ for the total coordinate ring of $\bb{P}V_1\times\cdots\times\bb{P}V_n$ with the usual $\bb{Z}^n$--grading. If $M$ is a finitely generated graded $S$--module and $\ul{a}\in\bb{Z}^n$, we write $M_{\ul{a}}$ for the $\ul{a}$--graded piece of $M$. We write $M(\ul{b})$ for the \defi{shifted module} given by $M(\ul{b})_{\ul{a}}=M_{\ul{a}+\ul{b}}$. If $\d=(d_1,\cdots,d_n)$ is a sequence of positive integers, we define the \defi{$\d$--syzygy modules} $K_{p,q}^{\d}(M)$ as the homology of (see also (\ref{eq:koszul}))
\[\bigwedge^{p+1}S_{\d}\oo M_{(q-1)\d}\to\bigwedge^{p}S_{\d}\oo M_{q\d}\to\bigwedge^{p-1}S_{\d}\oo M_{(q+1)\d}\]

\begin{theorem}\label{thm:reduction}
Let $q\geq 2$ be an integer and let $M$ be a finitely generated graded $S$--module. Consider the minimal free resolution of $M$
\begin{equation}\label{eq:resolution}
0\leftarrow M\leftarrow E_0\leftarrow E_1\leftarrow\cdots\leftarrow E_p\leftarrow\cdots\leftarrow E_m\leftarrow 0, 
\end{equation}
where
\[E_i=\bigoplus_{\ul{b}\in\mc{S}_i} F_{i,\ul{b}}\oo S(\ul{b}),\]
for some finite dimensional vector spaces $F_{i,\ul{b}}$, and finite subsets $\mc{S}_i\subset\bb{Z}^n$. If we let
\begin{equation}\label{eq:defPd}
 P(\d)=\min\{P_{q+i,\ul{b}}(\d)+i:\ul{b}\in\mc{S}_i,\ i=0,\cdots,m\},
\end{equation}
then we have
\[K_{p,q}^{\d}(M)=0\ \rm{for}\ p\leq P(\d).\]
\end{theorem}

\begin{proof} This follows from \cite[Prop.~(1.d.3)]{green1}. We sketch a proof for completeness. Consider the complex
\[F^{\bullet}:\ F^{-1}\to F^0\to F^1\to\cdots\to F^p\to 0,\]
where
\[F^i=\bigwedge^{p-i}S_{\d}\oo M_{(q+i)\d},\ i=-1,0,\cdots,p.\]
We have $H^i(F^{\bullet})=K_{p-i,q+i}^{\d}(M)$ for $i\geq 0$. We construct a double complex $G^{\bullet}_{\bullet}$ which is quasi--isomorphic to $F^{\bullet}$, by letting $G^i_j=\bigwedge^{p-i}S_{\d}\oo (E_j)_{(q+i)\d}$ for $i=-1,0,\cdots,p$, and $j=0,1,\cdots,m$: 
\[
 \xymatrix{
\bigwedge^{p+1}S_{\d}\oo (E_0)_{(q-1)\d} \ar[r] & \bigwedge^{p}S_{\d}\oo (E_0)_{q\d} \ar[r] & \cdots \ar[r] & S_{\d}\oo (E_0)_{(q+p-1)\d} \ar[r] & (E_0)_{(q+p)\d} \\
\bigwedge^{p+1}S_{\d}\oo (E_1)_{(q-1)\d} \ar[r]\ar[u] & \bigwedge^{p}S_{\d}\oo (E_1)_{q\d} \ar[r]\ar[u] & \cdots \ar[r]\ar[u] & S_{\d}\oo (E_1)_{(q+p-1)\d} \ar[r]\ar[u] & (E_1)_{(q+p)\d}\ar[u] \\
\vdots \ar[r]\ar[u] & \vdots \ar[r]\ar[u] & \ddots \ar[r]\ar[u] & \vdots \ar[r]\ar[u] & \vdots \ar[u] \\
\bigwedge^{p+1}S_{\d}\oo (E_m)_{(q-1)\d} \ar[r]\ar[u] & \bigwedge^{p}S_{\d}\oo (E_m)_{q\d} \ar[r]\ar[u] & \cdots \ar[r]\ar[u] & S_{\d}\oo (E_m)_{(q+p-1)\d} \ar[r]\ar[u] & (E_m)_{(q+p)\d}\ar[u] \\
}
\]
\noindent The vertical maps are induced from (\ref{eq:resolution}), while the horizontal ones are the usual Koszul differentials.

The vertical homology of $G^{\bullet}_{\bullet}$ is $F^{\bullet}$:
\[H_0(G^i_{\bullet})=F^i,\ H_j(G^i_{\bullet})=0\quad\rm{for}\ j>0,\]
while the horizontal homology of $G^{\bullet}_{\bullet}$ is given by
\[H^i(G^{\bullet}_j)=\bigoplus_{\ul{b}\in\mc{S}_j} F_{j,\ul{b}}\oo K_{p-i,q+i}^{\d}(S(\ul{b}))\]
for $i\geq 0$. Comparing the two spectral sequences associated to the double complex $G_{\bullet}^{\bullet}$ we conclude that in order to have that $K_{p,q}^{\d}(M)=0$ it suffices to show that $H^i(G^{\bullet}_i)=0$ for $0\leq i\leq m$, which in turn would be implied by the vanishing of the modules $K_{p-i,q+i}^{\d}(S(\ul{b}))$, for $0\leq i\leq m$, $\ul{b}\in\mc{S}_i$. Since $K_{p-i,q+i}^{\d}(S(\ul{b}))=0$ for $p-i\leq P_{q+i,\ul{b}}(\d)$, it follows from (\ref{eq:defPd}) that $K_{p,q}^{\d}(M)=0$ for $p\leq P(\d)$, concluding the proof of the theorem.
\end{proof}

\begin{corollary}
 Let $M$ be a finitely generated graded $S$--module. There exist integers $b_1,\cdots,b_n$ such that $K_{p,2}^{\d}(M)=0$ for all positive integers $d_i\geq -b_i$ and all $p\leq\min\{d_i+b_i:i=1,\cdots,n\}$. 
\end{corollary}

Consider now an arbitrary projective variety $X$. Given line bundles $\A_1,\cdots,\A_n$, we write $\mc{L}_{\d}(\A_1,\cdots,\A_n)$, or simply $\mc{L}_{\d}$ for $\A_1^{\oo d_1}\oo\cdots\oo \A_n^{\oo d_n}$. We have

\begin{corollary}\label{cor:reduction}
 Fix $q\geq 2$ and assume that $\A_1,\cdots,\A_n$ are very ample line bundles on $X$, sufficiently positive so that if we let $V_i=H^0(X,\A_i)$, then the natural maps
\begin{equation}\label{eq:surjectivity}
\Sym^{d_1}V_1\oo\cdots\oo\Sym^{d_n}V_n \lra H^0(\mc{L}_{\d}) 
\end{equation}
are surjective for all $d_i>0$. If $\mc{B}$ is any coherent sheaf on $X$, then there exist $m\geq 0$ and finite subsets $\mc{S}_i\subset\bb{Z}^n$, $i=0,\cdots,m$, such that if we define $P(\d)$ as in (\ref{eq:defPd}) then $K_{p,q}(X,\mc{B};\mc{L}_{\d})=0$ for $p\leq P(\d)$.

 In particular, let's assume that $n=1$ and that \cite[Conjecture~7.6]{ein-laz} holds. If $\mc{A}$ is a very ample line bundle on $X$ such that the corresponding embedding is projectively normal, then there exists a polynomial $P(d)$ of degree $(q-1)$ such that
\[K_{p,q}(X,\mc{B};\mc{A}^{\oo d})=0\ \rm{for}\ p\leq P(d).\]
\end{corollary}

\begin{proof} We have a commutative diagram where all arrows are closed embeddings
\[
\xymatrix{
  X \ar[r]^(.35){|\A_1|\times\cdots\times|\A_n|} \ar[d]_{|\mc{L}_{\d}|} & Y=\bb{P}V_1\times\cdots\times\bb{P}V_n \ar[d]^{|\mc{O}(\d)|} \\
  {\bb{P}H^0(\mc{L}_{\d}}) \ar@{^{(}->}[r] & \bb{P}W=\bb{P}(\Sym^{d_1}V_1\oo\cdots\oo\Sym^{d_n}V_n) \\
}
\]
so we can think of $\mc{B}$ as a sheaf on any of the spaces $X,Y,\bb{P}H^0(\mc{L}_{\d})$ or $\bb{P}W$. Since $\bb{P}H^0(\mc{L}_{\d})$ is a linear subspace of $\bb{P}W$ by (\ref{eq:surjectivity}), the syzygies of $\mc{B}$ on $\bb{P}W$ differ from those on $\bb{P}H^0(\mc{L}_{\d})$ by tensoring with a Koszul complex of linear forms. In particular
\[\min\{p:K_{p,q}(X,\mc{B};\mc{L}_{\d})\neq 0\}=\min\{p:K_{p,q}(Y,\mc{B};\mc{O}(\d))\neq 0\}.\]
We are then reduced to the case when $X=Y$ is a product of projective spaces and $\mc{L}_{\d}=\mc{O}(\d)$. The conclusion follows now from Theorem \ref{thm:reduction} if we let 
\[M=\bigoplus_{\ul{a}\in\bb{Z}_{\geq 0}^n} H^0(Y,\mc{B}\oo\mc{O}(\ul{a})).\]
\noindent The last part of the corollary follows from the fact that under the assumption of \cite[Conjecture~7.6]{ein-laz}, the functions $P_{q+i,b}(d)$ in (\ref{eq:defPd}) can be taken to be polynomials of degree $q+i-1\geq q-1$.
\end{proof}

\begin{corollary}
 Fix $q\geq 2$ and let $\mc{A}$ be an ample line bundle on a projective variety $X$. If $\mc{B}$ is any coherent sheaf on $X$, then there exist $m\geq 0$, finite subsets $\mc{S}_i\subset\bb{Z}^2$, $i=0,\cdots,m$, and functions $d_j=d_j(d)$, $j=1,2$, growing linearly with $d$, such that if we define $P(d_1,d_2)$ as in (\ref{eq:defPd}) and let $P(d)=P(d_1(d),d_2(d))$, then for sufficiently large values of $d$ we have
\[K_{p,q}(X,\mc{B};\mc{A}^{\oo d})=0\ \rm{for}\ p\leq P(d).\]
\end{corollary}

\begin{proof}
 Consider sufficiently large coprime integers $a_1,a_2$ such that $\mc{A}_i=\mc{A}^{\oo a_i}$, $i=1,2$ satisfy the hypotheses of Corollary \ref{cor:reduction}. If $d$ is large enough, we can find $d_1,d_2$ with $d_i\approx d/2a_i$ such that $d=a_1 d_1+a_2d_2$, i.e. $\mc{A}^{\oo d}=\mc{A}_1^{\oo d_1}\oo\mc{A}_2^{\oo d_2}$. The conclusion follows from Corollary \ref{cor:reduction}.
\end{proof}

\begin{corollary}
 Fix $q\geq 2$ and let $\mc{A},\mc{P}$ be line bundles on a projective variety $X$, with $\mc{A}$ ample. We let $\mc{L}_d=\mc{A}^{\oo d}\oo\mc{P}$. If $\mc{B}$ is any coherent sheaf on $X$, then there exist $m\geq 0$, finite subsets $\mc{S}_i\subset\bb{Z}^3$, $i=0,\cdots,m$, and functions $d_j=d_j(d)$, $j=1,2,3$, growing linearly with $d$, such that if we define $P(d_1,d_2,d_3)$ as in (\ref{eq:defPd}) and let $P(d)=P(d_1(d),d_2(d),d_3(d))$, then for sufficiently large values of $d$ we have
\[K_{p,q}(X,\mc{B};\mc{L}_d)=0\ \rm{for}\ p\leq P(d).\]
In particular, if the functions $P_{q,\ul{b}}(\d)$, $\ul{b}\in\bb{Z}^3$ grow as polynomials of degree $(q-1)$ in $d_1,d_2,d_3$, then $P(d)$ grows as a polynomial of degree $(q-1)$, so \cite[Conjecture~7.1]{ein-laz} holds.
\end{corollary}

\begin{proof}
 Consider sufficiently positive integers $a_1,a_2,a_3$ such that $\rm{gcd}(a_1,a_2+a_3)=1$ and such that the line bundles $\mc{A}_1=\mc{A}^{\oo a_1}$, $\mc{A}_2=\mc{A}^{\oo a_2}\oo\mc{P}$, $\mc{A}_3=\mc{A}^{\oo a_3}\oo\mc{P}^{-1}$ satisfy the hypotheses of Corollary~\ref{cor:reduction}. If $d$ is large enough, we can find $d_1,d_2$ with $d_1\approx d/2a_1$, $d_2\approx d/2(a_2+a_3)$ such that $d+a_3=a_1 d_1+(a_2+a_3)d_2$. If we let $d_3=d_2-1$ then $d=a_1d_1+a_2d_2+a_3d_3$ and
\[\mc{L}_d=\mc{A}^{\oo d}\oo\mc{P}=\mc{A}_1^{\oo d_1}\oo\mc{A}_2^{\oo d_2}\oo\mc{A}_3^{\oo d_3}.\]
The conclusion follows as before from Corollary \ref{cor:reduction}.
\end{proof}

\section*{Acknowledgments} 
I would like to thank David Eisenbud for his guidance during the early stages of this project, and Mihai Fulger, Yu--Han Liu, Giorgio Ottaviani, Greg Smith, Andrew Snowden and Kevin Tucker for inspiring conversations and for comments on earlier drafts of this paper. I am particularly indebted to Tom Church for introducing me to the theory of representation stability, and to Steven Sam for pointing me to the literature on chessboard and matching complexes, and for numerous helpful suggestions. I also thank Dan Grayson and Mike Stillman for making Macaulay2 \cite{M2}, which was useful in many experiments.


\end{document}